\documentclass[11pt, twoside]{article}

\usepackage{amsmath,amsthm}

\usepackage[utf8]{inputenc}

\usepackage{enumerate}

\usepackage{environ}
\NewEnviron{cases2}[1]{%
                      \begin{equation}\label{#1}\left\{ \begin{alignedat}{2} \BODY \end{alignedat}\right. \end{equation} }
\NewEnviron{cases3}[1]{%
                      \begin{equation}\label{#1} \begin{alignedat}{2} \BODY \end{alignedat} \end{equation} }
\NewEnviron{cases22}{%
                      \begin{equation}\left\{ \begin{alignedat}{2} \BODY \end{alignedat}\right. \end{equation} }

\usepackage{fancyhdr}
\usepackage{cite}
\usepackage{enumitem}

\usepackage{graphics}

\usepackage{algcompatible}
\usepackage{algorithm}

\usepackage{multicol}

\usepackage{xargs}                      
\usepackage[colorinlistoftodos,prependcaption,textsize=tiny]{todonotes}
\usepackage[normalem]{ulem}

\newcommandx{\pcomment}[2][1=]{\todo[linecolor=red,backgroundcolor=red!25,bordercolor=red,#1]{#2}}
\newcommandx{\kcomment}[2][1=]{\todo[linecolor=blue,backgroundcolor=blue!25,bordercolor=blue,#1]{#2}}

\newif\ifpreprint

\usepackage{hyperref}
\hypersetup{%
    colorlinks=True, 
    citecolor=blue,
    linkcolor=black}

\usepackage[charter]{mathdesign}

\newcommand{\rom}[1]{\uppercase\expandafter{\romannumeral #1\relax}}

\newcommand{\black}{\color{black}}

  \theoremstyle{definition}
  \newtheorem{theorem}{Theorem}[section]
  \newtheorem{corollary}[theorem]{Corollary}
  \newtheorem{proposition}[theorem]{Proposition}
  \newtheorem{lemma}[theorem]{Lemma}
  \newtheorem{definition}[theorem]{Definition}
  \newtheorem{remark}[theorem]{Remark}
   
  \newtheorem{example}[theorem]{Example}
  \newtheorem{examples}[theorem]{Examples}

  \newtheorem*{assumption*}{Assumption}

  \numberwithin{equation}{section}
    
\providecommand{\Div}{\operatorname{div}}          

\newcommand{\VN}{{\mathbf{N}}}

\newcommand{\VR}{{\mathbf{R}}}

\newcommand{\Dsf}{{\mathsf{D}}}

\providecommand{\Ca}{{\cal A}}

\providecommand{\Ch}{{\cal H}}

\setenumerate[0]{label=(\alph*)}

\newcommand{\eps}{{\varepsilon}}

\newcommand{\ben}{\begin{equation}}
\newcommand{\een}{\end{equation}}

\newcommand{\benn}{\begin{equation*}}
\newcommand{\eenn}{\end{equation*}}

\usepackage[colorinlistoftodos]{todonotes}

\newcommand{\Om}{{\Omega}}

\newcommand{\F}{{G}}

\usepackage{authblk}
\usepackage[T1]{fontenc}
\usepackage[utf8]{inputenc}

\usepackage{authblk}

\title{ The topological state derivative: an optimal control perspective on topology optimisation }

\author[1]{Phillip Baumann\thanks{E-Mail: phillip.baumann(at)tuwien.ac.at}}
\author[2]{Idriss Mazari-Fouquer\thanks{E-Mail: mazari(at)ceremade.dauphine.fr}}
\author[1]{Kevin Sturm\thanks{E-Mail: kevin.sturm(at)tuwien.ac.at}}

\affil[1]{TU Wien, Wiedner Hauptstr. 8-10,
      1040 Vienna, Austria}
\affil[2]{CEREMADE, UMR CNRS 7534, Universit\'e Paris-Dauphine, Universit\'e PSL, Place du Mar\'echal De Lattre De Tassigny, 75775 Paris cedex 16, France}

\begin{document}
\maketitle
\tableofcontents

\begin{abstract}
In this paper we introduce the topological state derivative for general topological dilatations and explore its relation to standard optimal control theory. We show that for a class of partial differential equations, the 
shape dependent state variable can be differentiated with respect to the topology, thus leading to 
a linearised system resembling those occurring in standard optimal control problems. However, a lot of care has to be taken when handling the regularity of the solutions of this linearised system. In fact, we should expect different notions of (very) weak solutions, depending on whether the main part of the operator or its lower order terms are being perturbed.  We also study the relationship with the topological state derivative, usually obtained through classical topological expansions involving boundary layer correctors. A feature of the topological state derivative is that it can either be derived via Stampacchia-type regularity estimates or alternately with classical asymptotic expansions.
 
  It should be noted that our approach is flexible enough to cover more than the usual case of point perturbations of the domain. In particular,  and in the line of \cite{a_DE_2022a,a_DE_2017a}, we deal with more general dilatations of shapes, thereby yielding topological derivatives with respect to curves, surfaces or hypersurfaces.  

 In order to draw the connection to usual topological derivatives, which are typically expressed with an adjoint equation, we show how usual first order topological derivatives of shape functionals can be easily computed using the topological state derivative.
\end{abstract}

\section{Introduction}
\subsection{Scope of the paper} The main goal in shape optimisation problems is to optimise a certain set, the ``design variable" $\Om$, in order to maximise or minimise a certain functional. To achieve this goal, it is necessary to understand how this functional varies under perturbations of $\Om$. Of particular importance are perturbations obtained by drilling a small inclusion $\omega_\eps$ of size $\eps$ into $\Om$. The first order variation of the functional under this perturbation is called the ``topological derivative". 
After its introduction in the pioneering works \cite{a_SOZO_1999a,a_GAGUMA_2001a} in the context of linear elasticity, the topological derivative framework was used in several numerical algorithms;  let us for instance mention level-set algorithms \cite{a_SOZO_1999a} or Newton-type algorithm \cite[Chapter 10]{b_NOSOZO_2019a}. 
We also refer to the monographs \cite{b_NOSO_2013a}, where several topological derivatives for various model problems are derived.

Recently \cite{c_DEST_2016a} a Lagrangian technique, called the ``averaged adjoint approach'',  was proposed as an efficient tool to compute topological derivatives. This technique allows for a wide range of applications: topological derivatives for Dirichlet boundary conditions \cite{a_AM_2021a}, topological derivatives for nonlinear  \cite{a_ST_2020a} and quasilinear problems \cite{a_GAST_2020a} or higher order topological derivatives  \cite{a_BAST_2021a} can be computed in a systematic way. We also refer to 
\cite{a_AM_2006a} for another Lagrangian technique to compute topological derivatives. 

In the even more recent paper \cite{a_DE_2022a},  a way to compute the topological derivative directly using the unperturbed adjoint equation was proposed.  In this reference, more general topological perturbations, called dilatations, are also considered; this leads to a more general notion of topological derivative.  In  \cite[Thm. 3.4]{a_DE_2022a},  the difference of the perturbed and unperturbed state variable are divided by the volume of the perturbation, however, no analysis on the existence of this limit is provided.  We will see that, for the models we consider, that the limit of the quotient divided by the volume of the perturbation for point perturbations and dilatations of hypersurfaces actually exists in a suitable function space; this leads us to a new notion of topological derivative of the state which we refer to as the {\emph topological state derivative}. A difference between \cite{a_DE_2022a} and our model problems is that in this reference homogenous Neumann boundary conditions on the inclusion boundary are imposed, while we deal with transmission problems, which can be seen as inhomogenous Neumann boundary conditions on the inclusion boundary. 

Our goal, in this paper, is to present a unique view on the topological derivative, by framing it as a usual derivative,  thereby leading to a direct approach to computing topological derivatives.
This is done by first perturbing the partial differential equation and then deriving a linearised equation as is usually done in optimal control theory \cite{b_ITKU_2008a,b_TR_2010a,b_HIULUL_2009a}. This shows that the design-to-state operator is actually differentiable for certain PDE constraints, and  that its derivative is described by a linearised system similar to optimal control problems. These linearised systems are usually very singular in the sense that their solutions admit low regularity. Typically, for problems where  the operator is perturbed, the linear system only admits very weak solutions. Interestingly these linearised systems may involve terms which are usually obtained from the classical asymptotic analysis performed on the problem under consideration.  Solutions of 
    the linearised system for the semilinear problem will be analysed in our paper through the notions introduced by Stampacchia, while the operator perturbation of the 
    transmission problem requires the notion of very weak solutions.  We remark that in state constrained optimal control problems low regularity of the adjoint equations is also an issue and thus the technical difficulties we encounter are related to the discussion of \cite{a_MEAPASC}, where the uniqueness of solutions to adjoint equations with mixed boundary conditions is discussed. Our approach also allows us to derive at least first order topological derivatives.

\paragraph{Structure of the paper}
 Our paper is structured as follows:
\begin{enumerate}
    \item In Section \ref{Se:BasicDefinition} we gather all the basic notions and definitions of generalised topological derivatives  and the topological state derivative.
\item Section \ref{Se:Co} contains a discussion of one of our main points, that is, the link between control derivatives, topological derivatives and the asymptotic analysis of PDEs. All the  rigorous computations in this section are carried out for linear operators, and serve to illustrate our idea.
\item Section \ref{Se:Ell} contains our rigorous results for the analysis of semilinear elliptic equations, when perturbing lower-order terms.  In Subsection \ref{Se:Adjoint}, we study several concrete examples using adjoint states.
\item Section \ref{Se:Op} is devoted to the study of point perturbations of the operator. The analysis is distinctly different from the semilinear case discussed in Section~\ref{Se:Ell}, both from the point of view of the notion of (very) weak solutions, and from that of first order asymptotics.
\item The rest of the paper contains the proofs of our results.
\end{enumerate}

\subsection{Generalised topological derivatives and the topological state derivative }\label{Se:BasicDefinition}

\paragraph{Generalised topological derivatives}
Throughout the paper, we let $\Dsf \subset \VR^d$ be a design region that is, a smooth, open, bounded domain.
Henceforth we denote by $\Ca(\Dsf)$ the set of admissible designs; in other words, 
\[ \Ca(\Dsf)=\left\{\Om \text{ measurable, }\Om \subset \Dsf\right\}.\]   A function $J:\Ca(\Dsf)\to \VR$ is called a shape functional. 
\begin{definition}\label{De:Main}
    Let $\Omega\in\Ca(\Dsf)$. Consider a compact set $E\subset \Dsf$ such that $\partial\Omega\cap E=\emptyset$ and denote by $E_\eps := \{x\in \VR^d:\; d_E(x)< \eps\}$  the tubular neighborhood of $E$ of width $\eps>0$. We define the perturbed set $\Omega(E_\eps)\subset\Dsf$ by
\begin{equation}
    \Om(E_\eps):=
 \begin{cases}
     \Om\cup E_\eps&   E\subset \Dsf\setminus \overline\Om, \\
     \Om\setminus \overline{E}_\eps &  E\subset \Om.
\end{cases}
\end{equation}
 The topological derivative of the functional $J$ at $E$ is defined by the following limit, provided it exists:
 \begin{equation}\label{def:topo_derivative}
        DJ(\Om)(E) := \lim_{\eps\searrow0} \frac{J(\Om(E_\eps)) - J(\Om)}{|E_\eps|}.
    \end{equation}
\end{definition}

\begin{remark}
    Here, we note that our definition of topological derivative already assumes that the first order term in the asymptotic expansion of $J$ is of order $|E_\eps|$,  the Lebesgue measure of $E_\eps$. This obviously depends on the shape functional under consideration. In several cases, for instance when considering a PDE dependent shape functional, and when enforcing Dirichlet boundary conditions on the boundary of $E_\eps$,  terms of lower order appear  \cite{a_AM_2021a}. However, as will be clear throughout, in all cases under consideration here, the leading order in the topological expansion is $|E_\eps|$. 
\end{remark}
Working with tubular neighborhoods allows for a great variety of perturbations; let us list a few examples corresponding to particular choices of $E$.
\begin{examples}
    Assume again that $\Dsf\subset \VR^d$ and $\Om\subset \Dsf$. 
    \begin{itemize}
        \item $E=\{x_0\}$, $x_0\in \Dsf$. Then $E_\eps = B_\eps(x_0)$ and $|E_\eps| = \eps^d|B_1(0)|$, where $B_r(x)$ denotes the open ball of radius $r>0$ located at $x$ in $\VR^d$.
                \item Let $E=\Gamma\subset \Dsf$ be a smooth closed orientable hypersurface with normal $\nu:\Gamma\to \VR^d$, $|\nu|=1$ on $\Gamma$. Then, for $\eps>0$ small enough, 
            $\Gamma_\eps = \{x+t\nu(x):\; x\in \Gamma:\; t\in [0,\eps)\}$ and $|\Gamma_\eps| = \eps \mathrm{Per}(\Gamma)+o_{\eps\to 0}(\eps),$ where $\mathrm{Per}(\Gamma)$ the perimeter of $\Gamma$,  which in view of the smoothness of $\Gamma$ is equal to the $(d-1)$-dimensional Lebesgue measure of $\Gamma$.    
    \end{itemize}
\end{examples}

\paragraph{The topological state derivative as derivative of the shape-to-state operator}
Throughout this paper we only consider PDE-dependent shape functionals. Let $X(\Dsf)$ be a space of functions defined on $\Dsf$ with values in $\VR$. We consider an equation of the type: find $u_\Om\in X(\Dsf)$, such that
\begin{equation}\label{eq:state_abstract}
    \langle E_\Om(u_\Om),\varphi\rangle_{X(\Dsf)',X(\Dsf)}=0 \quad \text{ for all } \varphi \in X(\Dsf),
\end{equation}
where $E_\Om:X(\Dsf)\to X(\Dsf)'$ is a potentially nonlinear operator. Typically, $X(\Dsf)$ is a Sobolev space ($X(\Dsf)=W^{1,p}(\Dsf)$), and \eqref{eq:state_abstract} merely corresponds to the weak formulation of an elliptic equation of the type 
\begin{equation}\label{eq:state_pragmatic}
\begin{cases}
\mathcal L_\Om u=f_\Om \quad\text{ in }\Dsf\,, 
\\ u\text{ satisfies boundary conditions on $\partial \Dsf$,}
\end{cases}
\end{equation}where the expression "weak formulation" needs to be specified. The operator $\mathcal L_\Om$ depends on $\Om$. In this paper, several dependences on $\Om$ are considered: $\mathcal L_\Om$ can take the form $-\Div((\alpha+\beta\chi_\Om)\nabla)$, or $-\Delta -\chi_\Om$, and can be nonlinear in $u$. Similarly, the function $f_\Om$ is \emph{a priori} assumed to depend on the set $\Om$.

\begin{definition}\label{De:StS}
    We define the shape-to-state operator $S:\Ca(\Dsf) \to X(\Dsf)$ by $S(\Om):= u_\Om$, where $u_\Om$ solves \eqref{eq:state_abstract} for the set $\Om\in \Ca(\Dsf)$.
\end{definition}
Of course, under proper assumptions on the nonlinear operator $E_\Om$, $S$ is a uniquely defined operator so that Definition \ref{De:StS} makes sense.

In the following definition we introduce the shape-to-state operator and its derivative, which we refer to as the topological state derivative. In contrast to the usual asymptotic expansion \cite[Chapter 5]{b_NOSO_2013a} of the state, our definition does not involve a rescaling and is simply the usual differential quotient of the state;  in this regard, it is akin to an optimal control approach.
\begin{definition}[Topological state derivative: derivative of the shape-to-state operator]\label{De:Main1}
Let $\Om\in \Ca(\Dsf)$ and consider a compact set $E\subset \Dsf\setminus\overline\Om$ or $E\subset \Om$. For $\eps >0$ we introduce
\begin{equation}
    U_\eps := U_{\Om(E_\eps)} := \frac{u_{\Om(E_\eps)} - u_{\Om}}{|E_\eps|},
\end{equation}
and define the \emph{topological state derivative} of $S$ at $\Om$ in direction $E$ by 
\begin{equation}
    S'(\Om)(E) := U_0:= U_{E,0} :=  \lim_{\eps\searrow 0} U_\eps, 
\end{equation}
where the limit has to be understood in an appropriate function space specified later on. 
\end{definition}

In the following sections we will examine three different PDE constraints and study the differentiability of the corresponding 
shape-to-state operator.  This will form the groundwork for the optimisation of several PDE constrained functionals.

\subsection{Control derivatives, topological  derivatives and asymptotic analysis}\label{Se:Co}
\paragraph{From control derivatives to topological derivatives}
When using the wording ``control derivative", what we mean is that the shape $\Om\in \Ca(\Dsf)$ is identified with its characteristic function $\chi_\Om$, and that we actually consider variations of $\Om$ as variations of $\chi_\Om$.  To give this concept a more precise meaning, let us take a basic example: for every $\Om \in \Ca$, let $u_\Om\in H^{1}(\Dsf)$ be the unique solution of 
\begin{cases22}
-\Delta u_\Om&=\chi_\Om\quad&&\text{ in }\Dsf,\\
 u_\Om&=0 \quad&&\text{ on }\partial\Dsf.
\end{cases22}
Let $E\subset \Dsf$ be either a point or a smooth oriented hypersurface, and assume for the sake of simplicity that $E\subset \Dsf\backslash \overline{\Om}$. Then, for $\eps>0$ small enough, we have, with $\Omega_\eps := \Omega(E_\eps)= \Omega\cup E_\eps$,
\[
\chi_{\Om_\eps}=\chi_{\Om}+\chi_{E_\eps},
\] 
so that, setting  $\mu_\eps=\frac{\chi_{E_\eps}}{|E_\eps|}$,  the function $U_\eps:=\frac{u_{\Om_\eps}-u_\Om}{|E_\eps|}$ solves 
\begin{cases2}{Eq:Main2}
    -\Delta U_\eps &=\mu_\eps &&\quad\text{ in }\Dsf,\\
    U_\eps& =0 &&\quad\text{ on } \partial\Dsf.
\end{cases2}
For each $\eps >0$ the function $\mu_\eps$ is a probability measure on $\Dsf$.
\black

In the case where $E=\{x_0\}$, it is clear that $\mu_\eps \rightharpoonup \delta_{x_0}$ as $\eps\searrow 0$ weakly in the sense of measures and it is then expected that $\{U_\eps\}_{\eps>0}$ converges in some sense to the solution  $U_{\{x_0\},0}\in X$ of the elliptic equation (with measure datum)
 \begin{cases2}{Eq:j}
        -\Delta U_{\{x_0\},0}&=\delta_{x_0}&& \quad \text{ in }\Dsf,\\ 
    U_{\{x_0\},0}&=0 && \quad \text{ on }\partial \Dsf.
\end{cases2}
To make the function space $X(\Dsf)$ precise, we will require some background information on the weak formulation of \eqref{Eq:j}, but what matters is that the topological state derivative appears, in this case, as the Green kernel of $-\Delta$. This simple remark allows to go back from topological derivatives to control derivatives.

\paragraph{Expressing control derivatives via the topological state derivative}
Indeed, assume we wish to compute the control derivative of \eqref{Eq:Main2}; this means that we see $\chi_\Om$ as a function in $L^2(\Dsf)$ and that we consider the control derivative of the state, defined, for a given perturbation $h$, as 
\[ 
    \dot u_h:=\lim_{t \searrow 0}\frac{v_{t h}-u_\Om}{t} 
\] where $v_{t h}\in L^2(\Dsf)$ satisfies \eqref{Eq:Main2} with $\chi_\Om$ replaced with $\chi_\Om+t h$. Then it is clear, by linearity of the equation, that $\dot u_h$ satisfies
\begin{cases22}
-\Delta \dot u_h&=h&&\quad\text{ in }\Dsf, 
\\ \dot u_h&=0&&\quad\text{ on }\partial\Dsf.
\end{cases22}
As we already explained briefly that the topological state derivative coincides with the Green kernel of the operator $-\Delta$, it is reasonable to expect, for instance if $h$ is supported in $\Dsf\backslash\overline\Om$, that $\dot u_h$ writes as 
\[ \dot u_h(x)=\int_{\Dsf\backslash \Om} h(y)U_{\{y\},0}(x)dy.\] Consequently, we see on this simple example that the knowledge of the topological derivative implies that we are able to compute any control-type derivative. One of our objectives in this paper is to prove the validity of this intuitive paradigm in several cases.

Of course, several points need to be underlined here. First and foremost, as should be clear, we need to work with elliptic equations with measure data in order to obtain optimal estimates. Most of this will be done using and adapting the techniques of \cite{Ponce}, which itself relies on the seminal  \cite{Littman} . Second, a lot of care needs to be taken when differentiating nonlinear problems, and giving proper regularity estimates on the fundamental solutions of the linearised operator; here, we  rely on the aforementioned \cite{Littman}. Finally, as we shall see, the weak formulation of the equation on $U_{\{x_0\},0}$ will be strongly dependent on the type of perturbation we consider. While, for perturbation of lower-order terms, the setting correspond to the standard one, we need to introduce a notion of very weak solution when considering transmission type problems.

\paragraph{Asymptotics of the shape-to-state operator of point perturbations}Our goal is now to link the control derivatives and the usual asymptotic analysis of the shape-to-state operator.

 ``Singular" perturbations (\emph{i.e.} removing a ball in the domain) and the asymptotics of PDEs where the singular perturbation appears are usually  treated by introducing so-called ``boundary layer correctors". This approach typically involves working in unbounded domains. Although working with an optimal control approach allows to only work in bounded domains, this limit layer approach is of great importance in topology optimisation and we thus present it in this paragraph.  We refer to \cite{b_MANAPL_2012a,b_MANAPL_2012b} for the asymptotic analysis of such singular perturbations and to \cite{b_NOSOZO_2019a,b_NOSO_2013a} for computations of topological derivatives of shape functionals. In contrast to these 
more classical approaches, we recall in this section the point of view of \cite{a_ST_2020a,a_BAGAST_2021b,a_BAST_2021a}, which, while also using boundary-layer correctors, rescales the domain to keep a fixed size of the inclusion.
As shown in \cite{a_GAST_2020a,a_ST_2020a}, this approach can be advantageous when dealing with semilinear and quasilinear PDEs.  For this reason we give the following definition:

\begin{definition}[Derivative of shape-to-state operator, the rescaled domain approach]\label{De:Main2}
Let $x_0\in \Dsf$. Define $E:=\{x_0\}$ and consider a connected and bounded domain $\omega\subset \VR^d$  with  $0\in \omega$.  For any $\eps>0$, we define the diffeomorphism $T_\eps:\Dsf \ni x\mapsto x_0+\eps x$ and the rescaled domain
\[\Dsf _\eps:=T_\eps^{-1}(\Dsf).\]
Introduce $\omega_\eps(x_0):= x_0 + \eps \omega$ and  define 
\begin{equation}\label{eq:omega_eps}
    \Omega_\eps(x_0,\omega) :=  \begin{cases}
     \Omega\cup \omega_\eps(x_0) &   \text{ for }x_0\in \Dsf\setminus \overline\Omega, \\
     \Omega\setminus \overline{\omega_\eps(x_0)} &  \text{ for } x_0 \in \Omega.
\end{cases}
\end{equation}
Note that according to Definition \ref{De:Main}, we have $\Om_\eps(x_0,\omega)=\Om(\omega_\eps(x_0))$. However, we introduce 
the notation $\Om_\eps(x_0,\omega)$ to emphasise the dependence on  both $x_0$ and $\omega$. Furthermore, we set $u_\eps := u_{\Om_\eps(x_0,\omega)}$, $u_0:= u_\Om$ and finally define
\[
K_\eps:=\frac{(u_\eps - u_0)\circ T_\eps}{\eps}, \quad \eps >0.
\]
The derivative of the shape-to-state operator is 
\begin{equation}\label{Eq:Back} 
     K:= \lim_{\eps\searrow 0} K_\eps,
\end{equation} 
where the limit has to be understood in an appropriate setting. We note that the limit $K$ typically depends on $x_0$, $\omega$ and as well as $\Omega$. As the domain of definition $\Dsf_\eps$ of $K_\eps$ varies with $\eps$, \eqref{Eq:Back} needs to be understood as $\Vert K_\eps-K\Vert_{X(\Dsf_\eps)}\to 0$ for the norm of a suitable function space $X(\Dsf_\eps)$.
\end{definition}

The function $K$ typically satisfies an equation in an unbounded domain. We refer to the later sections for examples and also to \cite{a_ST_2020a,a_BAGAST_2021b,a_BAST_2021a,a_GAST_2020b} for concrete topological derivative examples using the rescaling approach outlined above.

\begin{remark}
Let us underline that this definition covers the case of ball perturbations which corresponds to $\omega=B_1(0)$ in the previous definition, and is more general in the sense that shapes other than a ball are allowed. However, it does not include
lower dimensional objects. This is in contrast with Definition~\ref{De:Main1}.
\end{remark}

\paragraph{Connection between asymptotics of state and topological state derivative}
We consider again the problem of the previous section, namely,
    \begin{cases2}{eq:U_chi}
        -\Delta u_\Om&=\chi_\Om && \quad\text{ in }\Dsf,\\
        u_\Om&=0  && \quad\text{ on }\partial\Dsf.
\end{cases2}
We now sketch the connection between the asymptotic expansion of $u_{\Omega_\eps(x_0,\omega)}$ for $x_0\in \Dsf\setminus\overline{\Omega}$ and the topological state derivative. So we restrict ourselves to point perturbations and note that 
$\Omega_\eps(x_0,\omega) = \Omega\cup \{x_0 + \eps \omega\}$. We only discuss the case $d=3$ and provide the results for $d=2$ in later sections. In the fixed three dimensional domain $\Dsf$, with $f_1=1\,, f_2=0 $, we have \cite{a_BAGAST_2021b} the following expansion of $u_{\Omega_\eps(x_0,\omega)}=:u_\eps$:
\begin{equation}
    u_\eps(x) = u_0(x) + \eps^2(K(T_\eps^{-1}(x)) + \eps v(x)) + \text{ higher order terms}, \quad \text{ for a.e. } x\in \Dsf,
\end{equation}
where $u_0:=u_\Omega$ and $v$ is a regular boundary corrector function defined on the fixed domain $\Dsf$. Then in fact we will show that almost everywhere one can indeed recover the topological state derivative via the limit
\begin{equation}\label{eq:U0_limit}
    U_0(x)  = \lim_{\eps\searrow0} \frac{u_\eps-u_0}{|\omega_\eps|} = \lim_{\eps\searrow 0} \frac{1}{|\omega_\eps|}\eps^2(K(T_\eps^{-1}(x)) + \eps v(x)), \quad x\in \Dsf.
\end{equation}
The function $K$ admits the asymptotic behaviour $K(x) = R(x) + O(|x|^{-2})$ with $R(x):=   |\omega| E(x)$ and  $E(\cdot)$ being the fundamental solution of $-\Delta $ on $\VR^3$:
\begin{equation}
    E(x) = \frac{1}{4\pi|x|},
\end{equation} 
where here and henceforth we denote by $|x|$ the Euclidean norm of a vector $x\in \VR^d$. From the first asymptotic term of $K$, which is $R$, we can also determine the corrector $v\in H^1(\Dsf)$ as the solution of 
\begin{cases2}{}
  -\Delta v &=0 &&\quad\text{ in } \Dsf\,,
  \\ v(x) &= -R(x-x_0)&&\quad\text{ on }\partial \Dsf.
\end{cases2}
Therefore, one can compute the first limit on the right hand side of \eqref{eq:U0_limit} explicitly using $R(T_\eps^{-1}(x)) = \eps R(x-x_0)$:
\begin{equation}
    \lim_{\eps\searrow 0} \frac{1}{|\omega_\eps|}\eps^2(K(T_\eps^{-1}(x)) = \lim_{\eps\searrow 0} \frac{1}{|\omega_\eps|}\eps^2 R(T_\eps^{-1}(x)) = |\omega|^{-1}R(x-x_0) = \frac{1}{ 4\pi|x-x_0|}.
\end{equation}
We note that $x\mapsto R(x-x_0)\in W^{1,q}(\Dsf)$ for $q\in [1,\frac{d}{d-1}) = [1,\frac{3}{2})$. Summarising, we derived the following form of $U_0$:
\begin{equation}
    U_0(x) =  |\omega|^{-1}( R(x-x_0) + v(x)) \quad \text{ for a.e. } x\in \Dsf,
\end{equation}
and thus conclude that $U_0$ is indeed a solution of
\begin{cases2}{eq:limit_U0b}
    -\Delta U_0&=\delta_{x_0} && \quad \text{ in }\Dsf,  \\
    U_0 &=0   && \quad\text{ on } \partial \Dsf.
 \end{cases2}
 The solution $|\omega|^{-1}\left(R(x-x_0)+v(x)\right)$ is a well-known splitting for \eqref{eq:limit_U0b} and is often used in numerical analysis \cite{a_BEDELAMA_2018a,a_ER_1985a}. The function $|\omega|^{-1}R(x-x_0)$ solves the Poisson equation in $\VR^3$ with Dirac measure at $x_0$ as a right hand side, and $|\omega|^{-1}v(x)$ corrects the boundary error introduced by $|\omega|^{-1}R(x-x_0)$, so that $U_0$ has homogeneous Dirichlet boundary conditions on $\partial \Dsf$.  

 In conclusion, the topological state derivative can be obtained from the asymptotic analysis of the state equation. If the asymptotic analysis of the state equation is performed using compound asymptotics \cite{b_MANAPL_2012a,b_MANAPL_2012b,b_NOSO_2013a}, then one naturally obtains a splitting for the limit solution into a regular part, which comes from the corrector $v$ and an irregular part, which originates from the corrector $K$. We will see later that the topological expansion can be effectively used to compute the topological state derivative and even establish strong convergence in suitable function spaces.

\section{Main results for general topological perturbations of semilinear equations}\label{Se:Ell}
We first give some basic results about the convergence in measure of the functions $\chi_{E_\eps}$ (see Definition \ref{De:Main} for the definition of $E_\eps$).
In the following sections, we proceed with steps of increasing complexity, first considering topological state derivatives for lower order terms, then considering transmission problems.

\subsection{Convergence in measure of $\chi_{E_\eps}$ and notation}
Our goal is to make sense of topological derivatives for any type of $d$-dimensional inclusions as proposed in \cite{a_DE_2017a}. For this reason, we need to specify the behaviour of $\chi_{E_\eps}$, as $\eps\to 0$. This is the object of the following proposition; it is stated without a proof as it is fairly standard.

\begin{proposition}\label{Pr:Measure}
For every nonempty compact $E\subset \Dsf$ and for every $\eps>0$ we let $E_\eps$ be its tubular neighborhood (see Definition \ref{De:Main}) and we consider the probability measure  on $\Dsf$
\[
\mu_{E_\eps}:=\frac{\chi_{E_\eps}}{|E_\eps|}. 
\]
\begin{enumerate}
\item Assume $E=\{x_0\}$, so that $E_\eps=B_\eps(x_0)$. Then, for the weak convergence of measures, 
\[ 
\mu_{E_\eps}\underset{\eps\searrow 0}\rightharpoonup \mu_{E}:=\delta_{x_0}.
\]
\item Assume $E=\Gamma$ is a $(d-1)$-dimensional Lipschitz hypersurface with finite perimeter $\mathrm{Per}(\Gamma)$. Then, in the sense of measures, 
\[
    \mu_{E_\eps}\underset{\eps \searrow0}\rightharpoonup \mu_{E}:=\frac1{\mathrm{Per}(\Gamma)} (\Ch^{d-1}\lfloor \Gamma),
\]
where $\Ch^{d-1}\lfloor\Gamma$ stands for the restriction of the $(d-1)$-dimensional Hausdorff measure to $\Gamma$.
\item  Assume that $E=M$ is a $1< k < d-1$ dimensional compact and smooth submanifold (without boundary). Then, in the sense of measures:
\[
    \mu_{E_\eps}\underset{\eps \searrow0}\rightharpoonup \mu_{E}:=\frac1{\Ch^k(M)} (\Ch^k\lfloor M),
\]
\end{enumerate}
\end{proposition} 
We refer to \cite[Theorem 2.15]{a_DE_2022a} for a proof.

\begin{remark}
    \begin{itemize}
        \item For $E=\{x_0\}$ and $E_\eps=B_\eps(x_0)$ $\eps>0$ the weak convergence of measures (a) means for all $\varphi \in C^0(\overline{\Dsf})$:
        \begin{equation}
            \int_\Dsf \mu_{E_\eps} \varphi\;dx = \frac{1}{|B_\eps(x_0)|}\int_{B_\eps(x_0)} \varphi\;dx   \to \varphi(x_0) \quad \text{ as } \eps\searrow 0.
    \end{equation}
\item For $E=\Gamma$ is a $(d-1)$-dimensional Lipschitz hypersurface with finite perimeter $\mathrm{Per}(\Gamma)$, the weak convergence of measures (b) means for all $\varphi \in C^0(\overline{\Dsf})$:
\begin{equation}\label{eq:weak_conv_surf}
    \int_\Dsf \mu_{E_\eps} \varphi\;dx = \frac{1}{|E_\eps |}\int_{E_\eps} \varphi\;dx \to \frac1{\mathrm{Per}(\Gamma)}\int_{\Gamma} \varphi \; d\Ch^{d-1} \quad \text{ as } \eps \searrow 0.
\end{equation}
    \end{itemize}
\end{remark}

\begin{remark}
When $E$ is a hypersurface, we can actually prove that the convergence holds for the duality on $W^{1,p}(\Dsf)$. This means that \eqref{eq:weak_conv_surf} holds for every function $\varphi\in W^{1,p}(\Dsf)$, for $p\in [1,\infty)$, when we replace the last integral with $\int_\Gamma \mathrm{Tr}_\Gamma(\varphi)\;d\Ch^{d-1}$, where $\mathrm{Tr}_\Gamma$ is the trace operator on $\Gamma$.\black
\end{remark}
Throughout the paper, we retain the notation $\mu_{E}$ for the limit measures given in Proposition \ref{Pr:Measure}.

\begin{remark}[Lower dimensional inclusion]
Of course, what we considered here was the removal of a $d$-dimensional object: $E_\eps$ has nonempty interior. It is natural to wonder what might happen if we were to remove lower dimensional objects, for instance removing a centered disk in a three-dimensional object. We believe that our analysis would still be valid but, for the sake of readability, we stick with the removal of tubular neighborhoods.
\end{remark}

\paragraph{Notation}
In Definition \ref{De:Main} we have considered two types of perturbations, one consisting in adding some material outside of $\Om$, the other in removing some material from $\Om$. Naturally, this means that, depending on the case considered, either $\mu_{E_\eps}$ or $-\mu_{E_\eps}$ is involved in the linearised system. In order to alleviate notations and to not carry out moot distinctions, for every compact subset $E$ such that $E\subset \Om$ or $E\subset \Dsf \backslash \overline \Om$ we define
\begin{equation}\label{Eq:SignOm}
\mathrm{sgn}_\Om(E):=
\begin{cases}
    +1 &\text{ if }E\subset \Dsf\backslash \overline\Om\,, \\
    -1 &\text{ if } E\subset \Om.
\end{cases}
\end{equation}

\subsection{Topology optimisation problems for semilinear equations with monotone semilinearity}\label{se:sl}
\paragraph{Analytic setting}
The first problem we tackle is that of a semilinear elliptic equation, where $\Om\in \Ca(\Dsf)$ appears in the nonlinearity. 

 To be precise, we consider two coefficients $f_i\in\VR$ ($i=1,2$), as well as two nonlinearities $g_i=g_i(u)$ ($i=1,2$) that satisfy
\begin{equation}\label{Hyp:g}
g_i \text{ is $ C^1$ increasing in $u$, and is globally bounded in $\VR_+$ $(i=1,2).$ }
\end{equation}
We then define, for every $\Om\in \Ca(\Dsf)$, a nonlinearity $\rho_\Om=\rho_\Om(x,u)$ as 
\[ 
    \rho_\Om(x,u):=\chi_\Om(x)g_1(u)+\chi_{\Dsf\backslash\Om}(x)g_2(u).
\]
From \cite[Theorem 4.4]{b_TR_2010a}, if \eqref{Hyp:g} is satisfied, then, for every $\Om \in \Ca(\Dsf)$, the equation 
\begin{cases2}{Eq:MainSemiLinear}
-\Delta u_\Om+\rho_\Om(x,u_\Om)&=f_1\chi_\Om+f_2\chi_{\Dsf\backslash\overline\Om}&&\quad\text{ in }\Dsf, 
\\ u_\Om&=0&&\quad\text{ on }\partial\Dsf,
\end{cases2} 
has a unique solution $u_\Om \in H^1_0(\Om)$. By standard elliptic regularity, for every $p\in [1,\infty)$, $u_\Om\in W^{2,p}(\Dsf)$ so that $u_\Om\in C^1(\overline\Dsf)$. 
We now study the topological state derivative  of $\Om\mapsto u_\Om$. To give meaning to our afferent results, we need to lay down some basic definitions on the linearised system.

\paragraph{Basic computations}
Our subsequent analysis strongly hinges on the property of the linearised operator associated with \eqref{Eq:MainSemiLinear}. To justify the use of this linearisation, we simply observe that $U_\eps:=\frac{u_{\Om(E_\eps)}-u_\Om}{|E_\eps|}$ satisfies $U_\eps=0$ on $\partial \Dsf$ and in a weak $W^{1,q}_0(\Dsf)$-sense, the following equation in $\Dsf$:
\begin{multline*}
        -\Delta U_{\eps}+\chi_\Om\frac{(g_1(u_{\eps})-g_1(u_{0})-g_2(u_{\eps})+g_2(u_{0}))}{|E_\eps|}+\frac{g_2(u_{\eps})-g_2(u_0)}{|E_\eps|}\\
        =\mathrm{sgn}_\Om(E)\left[(g_2(u_{\eps})-g_1(u_{\eps}))+(f_1-f_2)\right]\mu_{E_\eps},
\end{multline*}
 where we used the simplified notation $u_\eps:=u_{\Omega(E_\eps)}$ and $u_0:=u_\Om$, and we should thus obtain, as $\eps\searrow 0$, the following limit equation:
 \begin{cases3}{}
     -\Delta U_0+\frac{\partial\rho_0}{\partial u}(x,u_0)& =\mathrm{sgn}_\Om(E)\left[ (g_2(u_0)-g_1(u_0))+(f_1-f_2)\right]\mu_E &&  \quad \text{ in } \Dsf, \\
     U_0 & =0 && \quad \text{ on } \partial \Dsf,
\end{cases3}
which has to be understood in a weak $W^{1,q}(\Dsf)$ sense for $q>d$ and will be explained in the next paragraph. In the following paragraph we give some background information about the linear operator used to define the linear equation on $U_0$.

\paragraph{Notion of weak solution for the linearised system}
The linearised operator associated with \eqref{Eq:MainSemiLinear} is defined as 
\begin{equation}\label{Eq:LSemiLinear} \mathcal L_\Om:u\mapsto -\Delta u+\frac{\partial \rho_\Om}{\partial u}(x,u_\Om)u.\end{equation}As we explained in Section \ref{Se:BasicDefinition}, topological state derivatives ``should", in a sense made precise below, solve an equation of the form $-\mathcal L_\Om u=\mu$ for some probability measure $\mu$, with homogeneous Dirichlet boundary conditions. Even in the case of the Laplacian there are natural Sobolev bounds on the regularity to be expected from solutions of such equations. This motivates the following definition.

\begin{definition}\label{De:WeakSolutionSemiLinear}
Let $\mathcal M(\Dsf)$ be the set of Borel measures in $D$ with finite total variation. Let $\mu \in \mathcal M(\Dsf)$. For every $q\in [1,\frac{d}{d-1})$, we say that a function $u\in W^{1,q}_0(\Dsf)$ is a weak $W^{1,q}_0$-solution of 
\begin{cases2}{Eq:MainLinearisedSemiLinear}
    \mathcal L_\Om u& =\mu && \quad\text{ in }\Dsf, \\ 
u& =0 && \quad\text{ on }\partial \Dsf,
\end{cases2}
if, for every function $\varphi\in W^{1,q^\prime}_0(\Dsf)$ with $\frac{1}{q}+\frac{1}{q^\prime}=1$, there holds
\begin{equation}\label{eq:radon_measure}
\int_\Om  \nabla u\cdot\nabla \varphi\; dx+\int_\Om \frac{\partial \rho_\Om}{\partial u}(x,u_\Om) u\varphi\;dx=\langle \varphi,\mu\rangle,
\end{equation}
where the last duality bracket is to be understood in the sense of the duality between continuous functions and measures.
\end{definition}
It should be noted that since $1\le q<\frac{d}{d-1}$, the conjugate Lebesgue exponent $q^\prime$ of $q$, $\frac{1}{q}+\frac{1}{q^\prime}=1$, satisfies $q^\prime>d$. From Sobolev embeddings this implies that the duality bracket 
$\langle \varphi,\mu\rangle$ in \eqref{eq:radon_measure} is well-defined. Definition \ref{De:WeakSolutionSemiLinear} is a standard notion of weak solution for elliptic equations with measure data \cite{Boccardo,Ponce}.

As a first consequence of \eqref{Hyp:g} we prove that  \eqref{Eq:MainLinearisedSemiLinear} is well-posed.
\begin{proposition}\label{Pr:LinearisedSemiLinear}
If $g_1\,, g_2$ satisfy \eqref{Hyp:g} then, for every finite Borel measure $\mu$ in $\Dsf$, the equation \eqref{Eq:MainLinearisedSemiLinear} is well-posed: for every $q\in [1,\frac{d}{d-1})$ there exists a unique solution $u\in W^{1,q}_0(\Dsf)$ of \eqref{Eq:MainLinearisedSemiLinear}. Furthermore, there exists a constant $C_q$ independent of $\mu$ such that 
\[
\Vert u\Vert_{W^{1,q}(\Dsf)}\leq C_q \Vert \mu\Vert_{\mathcal M(\Dsf)}.
\]
\end{proposition}
It is obvious by the inclusion of the Lebesgue spaces $L^p$ that $u$ does not depend on the exponent $q$ and we abbreviate the first point of this proposition as ``there exists a unique weak solution $u$ to \eqref{Eq:MainLinearisedSemiLinear} that further satisfies that for every $q\in [1,\frac{d}{d-1})$, $u \in W^{1,q}_0(\Dsf)$.'' Finally, observe that the Sobolev space in which \eqref{Eq:MainLinearisedSemiLinear} is well-posed depends on the space $\mu_E$ is in the dual of; in the case where $E=\{x_0\}$, $\mu_E$ is only in the dual space of $W^{1,q}(\Dsf)$, $q>d$. In the case $E=\Gamma$, $\mu_E$ is in the dual of all Sobolev spaces by the theory of Sobolev traces.

\paragraph{Expression of the topological state derivative}
Our main result here is the following theorem (recall that $\mu_{E}$ is defined in Proposition \ref{Pr:Measure})
\begin{theorem}\label{Th:MainSemiLinear}
 Let $E\subset \Om$ or $E\subset \Dsf\backslash \overline{\Om}$ be either a point, $d-1$-dimensional Lipschitz surface or $1<k<d-1$ dimensional compact and smooth submanifold (without boundary). For every $\eps>0$ we define $U_\eps:=\frac{u_{\Om(E_\eps)}-u_\Om}{|E_\eps|}$. Then, for every $q\in [1,\tfrac{d}{d-1})$,
\[
U_\eps\underset{\eps\searrow 0}\rightarrow U_{0}\text{ strongly in $W^{1,q}_0(\Dsf)$,}
\] 
where $U_{0}$ is the unique solution to 
\begin{cases2}{Eq:UeSemiLinear}
-\Delta U_{0}+\partial_u\rho_\Om(x,u_\Om)U_{0}&=\mathrm{sgn}_\Om(E)\left\{(g_2(u_\Om)-g_1(u_\Om))+(f_1-f_2)\right\}\mu_{E}&&\quad\text{ in }\Dsf, 
\\ U_{0}&=0&&\quad\text{ on }\partial \Dsf.
\end{cases2}
\end{theorem}
Observe that, since $u_\Om\in C^1(\overline\Dsf)$ and since $g_1\,, g_2$ are continuous, the product appearing on the right-hand side of \eqref{Eq:UeSemiLinear} is indeed a Borel measure. In addition, $U_0$ depends on $\Omega$ and $E$.


\paragraph{Expression of control derivatives using the topological state derivative}
We mentioned in the introduction of this paper that a control point of view allows to obtain the topological state derivative. Conversely, in low dimensions, the knowledge of the topological state derivatives enables the recovery of control derivatives. In this context,  and to make our statement more precise, let us recall that, seeing $\chi_\Om$ as a function in $L^2(\Dsf)$, we may extend the definition of $u_\Om$ by defining, for every $f\in L^2(\Dsf)$, $u_f$ as the unique solution of \eqref{Eq:MainSemiLinear} with $\chi_\Om$ replaced with $f$, and $\chi_{\Dsf\backslash\overline\Om}$ replaced with $(1-f)$. The $L^2$-differentiability of the map $f\mapsto u_f$ is standard. For every $h\in L^2(\Omega)$, let 
\[
    \dot u_{\chi_\Om,h}=\lim_{t\searrow0}\frac{u_{\chi_\Om+th}-u_{\chi_\Om}}{t}
\]
be the directional derivative of $f\mapsto u_f$ at $\chi_\Om$ in direction $h$. Recalling that $\mathcal L_\Om$ was defined in \eqref{Eq:LSemiLinear}, $\dot u_{\chi_\Om,h}$ solves (in the weak $H^{1}_0(\Dsf)$  sense):
\begin{cases2}{Eq:Udot}
\mathcal L_\Om \dot u_{\chi_\Om,h}&=h\left\{(g_2(u_\Om)-g_1(u_\Om))+(f_1-f_2)\right\}&&\quad\text{ in }\Dsf,
\\ \dot u_{\chi_\Om,h}&=0&&\quad\text{ on } \partial\Dsf.
\end{cases2}

Our second theorem is the following:
\begin{theorem}\label{Th:Main2SemiLinear}
Assume $d\in\{2,3\}$. Then, for every $h\in L^2(\Dsf)$, $\dot u_{\chi_\Om,h}$ admits the following representation: for a.e. $x\in \Dsf$,
\begin{equation}\label{Eq:GKSemiLinear}
    \dot u_{\chi_\Om,h}(x)=\int_{\Dsf} \mathrm{sgn}_\Om(y)U_{\{y\},0}(x)h(y)dy.
\end{equation}
where $U_{\{y\},0}$ is the solution of \eqref{Eq:MainLinearisedSemiLinear} with $\mu=\delta_y$.
\end{theorem}
Theorem~\ref{Th:Main2SemiLinear} justifies the analogy between the topological state derivative and the Green kernel of the linearised operator.  Of course, there is an interplay between the dimension assumption $d\in \{2,3\}$ and the integrability of the perturbation $h$. We also note here that we stated the theorem in the $L^2$ setting, as it is the most currently used for control derivatives.

\subsection{Asymptotic expansion of $u_\Omega$ and relation to the topological state derivative}\label{subsec:asymptotic_rhs}
\paragraph{Asymptotic analysis in the linear case}
In this section, we consider the model \eqref{Eq:MainSemiLinear} with $g_1=g_2=0$ for point perturbations. To be precise we consider $u_\eps\in H^1_0(\Dsf)$, such that
\begin{equation}\label{eq:perturbed_state_rhs}
    \int_\Dsf \nabla u_\eps \cdot  \nabla \varphi \;dx =\int_\Dsf (f_1 \chi_{\Omega_\eps} + f_2\chi_{\Omega_\eps^c})\varphi \;dx \quad\text{ for all } \varphi \in H^1_0(\Dsf). 
\end{equation}
where 
\begin{align}
    \Omega_\eps & := \Omega_\eps(x_0,\omega) :=  \begin{cases}
     \Omega\cup \omega_\eps(x_0) &   \text{ for }x_0\in \Dsf\setminus \overline\Omega, \\
     \Omega\setminus \overline{\omega_\eps(x_0)} &  \text{ for } x_0 \in \Omega,
\end{cases}  
\end{align}
and $\omega_\eps(x_0):= x_0 + \eps \omega$ with $\omega\subset \VR^d$ being a simply connected domain with $0\in \omega$. Furthermore, we define $u_\eps := u_{\Om_\eps(x_0,\omega)}$ for $\eps >0$ and $u_0:= u_\Om$.

The full asymptotic expansion for this equation including full topological expansions of several cost functionals
has been studied in \cite{a_BAGAST_2021b}. Recall the notation $T_\eps(x) = x_0 + \eps x$ for $x_0\in \Dsf\setminus\partial \Omega$. We now present a relation between the limit
\begin{equation}\label{eq:limit_x_0}
    U_0 := \lim_{\eps\searrow 0} \frac{u_\eps-u_0}{|\omega_\eps|},
\end{equation}
and the asymptotic expansion derived in \cite{a_BAGAST_2021b} for $u_\eps$, namely,
\begin{equation}\label{eq:point_wise_d2}
    u_\eps(x) = u_0(x) + \eps^2( K(T_\eps^{-1}(x)) + v(x) + \ln(\eps)b) +o(\eps^2) \quad \text{ for } d=2,
\end{equation}
with $b:= -\mathrm{sgn}_\Om(\{x_0\})(2\pi)^{-1}(f_1-f_2)$ and
\begin{equation}\label{eq:point_wise_d3}
    u_\eps(x) = u_0(x) + \eps^2( K(T_\eps^{-1}(x)) + \eps^{d-2}v(x)) + o(\eps^d) \quad \text{ for } d\ge 3.
\end{equation}
Here $K$ is a  corrector function defined in $\VR^d$, while $v$ is a boundary layer corrector defined in the bounded domain $\Dsf$. 
To be precise, $K$  is given, in term of the fundamental solution $E(\cdot)$ of the Laplace operator $-\Delta$ in $\VR^d$, as
\begin{equation}\label{eq:K_rhs}
    K(x) = \mathrm{sgn}_\Om(\{x_0\})(f_1-f_2)\int_\omega E(x-y)\; dy,
\end{equation}
with the fundamental solution being given by
\begin{equation}\label{eq:fundamental_solution}
    E(x) = \begin{cases}
        c_2 \ln(|x|) & \text{ for } d=2, \\
        c_d\frac{1}{|x|^{d-2}} & \text{ for } d\ge 3,
    \end{cases}
\end{equation}
with $c_2:=-\frac{1}{2\pi}$ and, if $d\geq 3$, $c_d:= ((d(d-2)\alpha(d))^{-1}$, $\alpha(d)$ denoting the volume of the unit ball in $\VR^d$. Thus, the function $K$ satisfies
\[
    -\Delta K = \mathrm{sgn}_\Om(\{x_0\})(f_1-f_2)\chi_\omega \quad \text{ in } \VR^d.
\]
and admits the following asymptotic expansion for $d\ge 2$:
\begin{equation}\label{eq:asymptotic_K_rhs}
    K(x)  =  \mathrm{sgn}_\Om(\{x_0\})(f_1-f_2)\begin{cases}
    c_2 \ln(|x|) + O(|x|^{-1}) & \text{ for } d=2, \\
    c_d |x|^{-(d-2)}  + O(|x|^{-(d-1)})& \text{ for } d\ge 3,
    \end{cases}
\end{equation}
so that when we denote by $R$ the first term of the asymptotics of $K$, that is, $K(x)  = R(x) + O(|x|^{d-1})$, we obtain
\begin{equation}\label{eq:R_rhs}
   R(x) =  \mathrm{sgn}_\Om(\{x_0\})(f_1-f_2)\begin{cases}
    c_2 \ln(|x|)  & \text{ for } d=2, \\
    c_d |x|^{-(d-2)}  & \text{ for } d\ge 3.
    \end{cases}
\end{equation}
The corrector function $v\in H^1(\Dsf)$ satisfies 
        \begin{cases2}{}
            -\Delta v & =0  &&  \quad \text{ in } \Dsf,\\
            v(x) &= -R(x-x_0) && \quad \text{ on } \partial \Dsf. \label{eq:vk}
    \end{cases2}

    To state the main result, let us briefly recall the setting of \cite{a_BAGAST_2021b}. Since the result in \cite{a_BAGAST_2021b} was provided only for $d\in \{2,3\}$, we give  a short proof in the appendix. 

\begin{lemma}\label{lma:expansion_rhs}
    Introduce the function
\begin{equation}
    K_\eps = \frac{(u_\eps-u_0)\circ T_\eps}{\eps^2}, \quad \eps>0.
\end{equation}
Set $\Dsf_\eps := T_\eps^{-1}(\Dsf)$ for $\eps>0$. Then there is a constant $C=C_{p,d}>0$, which depends on $p$ and $d$, such that for $d=2$:
\begin{equation}\label{eq:main_2}
    \|\eps(K_\eps - K - v\circ T_\eps-\ln(\eps)b)\|_{L^2(\Dsf_\eps)} + \|\nabla(K_\eps - K - v\circ T_\eps-\ln(\eps)b)\|_{L^2(\Dsf_\eps)^d} \le C \eps, 
\end{equation}
with $b:=-\mathrm{sgn}_\Om(\{x_0\}) \frac{1}{2\pi}(f_1-f_2)$ and for $d\ge 3$:
\begin{equation}\label{eq:main_3}
    \|\eps(K_\eps - K - \eps^{d-2}v\circ T_\eps)\|_{L^2(\Dsf_\eps)} + \|\nabla(K_\eps - K - \eps^{d-2}v\circ T_\eps)\|_{L^2(\Dsf_\eps)^d} \le C \eps^{\frac{d}{2}}.
\end{equation}
\end{lemma}
\paragraph{Relation between asymptotic expansion and topological state derivative}
To draw a connection to the topological state derivative, let us note that we can write the expansions \eqref{eq:main_2} and \eqref{eq:main_3} pointwise as \eqref{eq:point_wise_d2} and \eqref{eq:point_wise_d3}. 
Our main result is that the estimate \eqref{eq:main_3} in fact implies the following estimate:
\begin{theorem}\label{thm:main_rates_rhs}
\begin{itemize}
For $x_0\in \Dsf\setminus \partial \Omega$ and $\omega\subset \VR^d$ be a simply connected and bounded domain with $0\in \omega$ we use the definition of $\Omega_\eps(x_0,\omega)$ of \eqref{eq:omega_eps} and set $u_\eps := u_{\Omega_\eps(x_0,\omega)}$, $u_0:= u_\Omega$ and
\begin{equation}
    U_\eps := \frac{u_\eps - u_0}{|\omega_\eps|}, \quad \eps >0. 
\end{equation}
Let $K$ and $v$ be defined by \eqref{eq:K_rhs} and \eqref{eq:vk}, respectively. Then we have the following results.
\item[(i)]
    From the asymptotic expansion \eqref{eq:main_2} and \eqref{eq:main_3} we conclude that 
    the limit \eqref{eq:limit_x_0} exists. In fact we have
 \begin{equation}
     U_0(x) =  \mathrm{sgn}_\Om(\{x_0\})(f_1-f_2) E(x-x_0) + v(x), \quad \text{ for a.e. } x\in \Dsf,
\end{equation}
and for a.e. $x\in \Dsf$, we have with $b:= -\mathrm{sgn}_\Om(\{x_0\})(2\pi)^{-1}(f_1-f_2)$:
\begin{equation}\label{eq:pointwise_limit_rhs_dim2}
    U_0(x) = \lim_{\eps\searrow 0}\frac{1}{|\omega_\eps|} \eps^2(K(T_\eps^{-1}(x)) + v(x) + \ln(\eps)b) \quad \text{ for } d=2,
\end{equation}
and 
\begin{equation}\label{eq:pointwise_limit_rhs_dim3}
    U_0(x) = \lim_{\eps\searrow 0}\frac{1}{|\omega_\eps|}  \eps^2 (K(T_\eps^{-1}(x)) + \eps^{d-2} v(x)) \quad \text{ for } d\ge 3.
\end{equation}
\item[(ii)] Let the space dimension be $d=2$.   Then there exist constant $C=C_{p,d}>0$, which depends on $p$ and $d$, such that
    \begin{itemize}
        \item[$\bullet$] We have for all $p\in (2,\infty)$:
\begin{equation}
    \|U_\eps - U_0\|_{L^p(\Dsf)} \le C\eps^{\frac2p}.
\end{equation}
\item[$\bullet$] We have for all $p\in (1,2)$:
    \begin{equation}\label{eq:estimate_W1p_dim2}
    \|U_\eps - U_0\|_{W^{1,p}(\Dsf)} \le  C\eps^{\frac2p-1}.
\end{equation}
\end{itemize}
\item[(iii)] Let the space dimension be $d\ge 3$. Then there exists $C=C_{p,d}>0$, which also depends on $p$ and $d$, such that:
    \begin{itemize}
        \item[$\bullet$]  We have for all $p\in \left(\frac{d}{d-1},\frac{d}{d-2}\right)$:
    \begin{equation}
        \|U_\eps - U_0\|_{L^p(\Dsf)} \le C\eps^{\frac{d-p(d-2)}{p}}.
    \end{equation}
        \item[$\bullet$] We have for all $p\in (1,\frac{d}{d-1})$,
            \begin{equation}\label{eq:rate_3D_linear}
        \|U_\eps - U_0\|_{W^{1,p}(\Dsf)} \le C\eps^{\frac{d-p(d-1)}{p}}.  
    \end{equation}
    Notice that $p^{-1}(d-p(d-1)) \in (0,1)$.
    \end{itemize}
\end{itemize}
\end{theorem}

The following corollary shows that, if $\omega=B_1(0)$ is the unit ball centered at the origin, then the convergence rate of $U_\eps$ to $U_0$ in the $L^p(\Dsf)$ norm can be improved to an order between $(1,2)$. 

\begin{corollary}\label{cor:symmetric_inclusion}
Assume that $\omega=B_1(0)$ is the unit ball in $\VR^d$ centered at the origin. Then we have for all $p\in (1,\frac{d}{d-2})$ for $d\ge3$ or $p\in (1,\infty)$ for $d=2$,
\begin{equation}\label{eq:lp_rhs_convergence}
        \|U_\eps - U_0\|_{L^p(\Dsf)} \le C\eps^{\frac{d-p(d-2)}{p}}.  
    \end{equation}
    That means the convergence rates \eqref{eq:estimate_W1p_dim2} for $d=2$ and \eqref{eq:rate_3D_linear} for $d\ge3$ are improved for all $p\in (1,\frac{d}{d-1})$.
\end{corollary}

Let us finish this section with two remarks.
\begin{remark}
    The function $U_0(x):= |\omega|^{-1}(R(x-x_0) + v(x))\in W^{1,p}_0(\Dsf)$ with $p\in [1,\frac{d}{d-1})$ solves in a weak sense:
    \begin{cases2}{eq:limit_dirac}
        -\Delta U_0  & = \mathrm{sgn}_\Om(\{x_0\})(f_1-f_2) \delta_{x_0}  && \quad  \text{ in } \Dsf,  \\
        U_0  & =0  && \quad \text{ on } \partial \Dsf.
    \end{cases2}
    The decomposition of the solution $U_0$ of \eqref{eq:limit_dirac} in a regular part $|\omega|^{-1}v$ and a singular part $|\omega|^{-1}R(x-x_0) = \mathrm{sgn}_\Om(\{x_0\})(f_1-f_2)|\omega|  E(x-x_0)$ is well-known and often used in the numerical investigation of this type of equation \cite{a_ER_1985a,a_BEDELAMA_2018a}.
\end{remark}

\begin{remark}
The improved $L^p(\Dsf)$ convergence rate of Corollary~\ref{cor:symmetric_inclusion} is a result of the symmetry of the 
inclusion $\omega=B_1(0)$. 
    We note that \eqref{eq:lp_rhs_convergence} implies for $p>1$ close to one that 
\begin{equation}
    \|(U_\eps - U_0)/\eps\|_{L^p(\Dsf)} = o(1)
\end{equation}
and thus $U_\eps^2 := (U_\eps - U_0)/\eps \to 0$ strongly in $L^p(\Dsf)$ for $p$ close to one. This is actually consistent with the limit equation of $U_\eps^2$. To see this we recall that 
\[
    U_\eps=\frac{u_\eps-u_0}{|\omega_\eps|}, \quad U_\eps^2 := \frac{U_\eps-U_0}{\eps}, \quad \eps >0.
\]
It is readily checked that $U_\eps^2$ satisfies:
\begin{equation}\label{eq:convergence_ball_lp2}
    \int_\Dsf \nabla U_\eps^2\cdot  \nabla \varphi \;dx = \text{sgn}_\Omega(\{x_0\}) (f_1-f_2)\frac{1}{|\omega_\eps|} \int_{\omega_\eps} \eps^{-1}(\varphi - \varphi(x_0)) \;dx \quad \text{ for all } \varphi \in H^1_0(\Dsf). 
\end{equation}
Now changing variables on the right hand side and integrating by parts on the left hand side, we obtain for $\varphi \in C^2_c(\Dsf)$:
\begin{equation}
    -\int_\Dsf U_\eps^2 \Delta \varphi \;dx =  \text{sgn}_\Omega(\{x_0\}) (f_1-f_2) \int_{\omega} \eps^{-1}(\varphi(x_0+\eps x) - \varphi(x_0)) \;dx.
\end{equation}
Hence, if $U_\eps^2\to U_0^2$ in $L_p(\Dsf$), then passing to the limit yields for $\varphi \in C^2_c(\Dsf)$:
\begin{equation}
    -\int_\Dsf U^2_0 \Delta \varphi \;dx = \text{sgn}_\Omega(\{x_0\}) (f_1-f_2) \int_\omega \nabla \varphi(x_0)\cdot x\;dx.
\end{equation}
So we observe that for $\omega=B_1(0)$ the integral on the right hand side vanishes due to the symmetry of $B_1(0)$. This is consistent with $U_0^2=0$ so that also the left hand side is zero. 
\end{remark}

\paragraph{Topological state derivative via the formal asymptotic expansion of the semilinear equation}
 We consider the semilinear equation \eqref{Eq:MainSemiLinear} with right hand side $f\in L^2(\Dsf)$:
\begin{cases2}{Eq:MainSemiLinear2}
-\Delta u_\Om+\rho_\Om(x,u_\Om)&=f &&\quad\text{ in }\Dsf, 
\\ u_\Om&=0&&\quad\text{ on }\partial\Dsf.
\end{cases2} 
From Theorem~\ref{Th:MainSemiLinear} we have that the limit
\begin{equation}\label{eq:limit_semilinear}
    U_0 := \lim_{\eps\searrow 0 }\frac{u_\eps - u_0}{|\omega_\eps|},
\end{equation}
satisfies
\begin{cases2}{eq:limit_U0_semilinear_asympt1}
    -\Delta U_{0}+\partial_u\rho_\Om(x,u_\Om)U_{0}&=\mathrm{sgn}_\Om(\{x_0\})\delta_{x_0}\left\{(g_2(u_\Om)-g_1(u_\Om))\right\} &&\quad\text{ in }\Dsf,
\\ U_{0}&=0&&\quad\text{ on }\partial \Dsf.
\end{cases2}
We now want to show that this limit can also be obtained using an asymptotic analysis of $u_\Om$. For this purpose, we can split the solution $U_0$ into an irregular part $|\omega|^{-1}R$ and regular part $|\omega|^{-1}v$ as follows (the factor 
$|\omega|^{-1}$ is chosen to make the link between the asymptotic expansion and will become clear shortly). Set 
$g_{x_0} := g_2(u_\Om(x_0))-g_1(u_\Om(x_0))$ and let $R$ be defined by 
\begin{equation}
    R(x) =  \mathrm{sgn}_\Om(\{x_0\}) g_{x_0}|\omega|\begin{cases}
     c_2 \ln(|x|)  & \text{ for } d=2, \\
     c_3 |x|^{-1}  & \text{ for } d= 3.
    \end{cases}
\end{equation}
Then we have in a distributional sense:
\begin{equation}\label{eq:R}
    -\Delta (R(x-x_0)) = |\omega|  \mathrm{sgn}_\Om(\{x_0\}) g_{x_0} \delta_{x_0} \quad \text{ in } \VR^d.
\end{equation}
Now we define $v\in H^1(\Dsf)$ as follows
\begin{cases2}{eq:vk_semilinear1}
    -\Delta v+ \partial_u\rho_\Om(x,u_\Om) v &= -\mathrm{sgn}_\Om(\{x_0\})R(x-x_0)\partial_u\rho_\Om(x,u_\Om)   &&\quad\text{ in }\Dsf,\\ 
    v&=-R(x-x_0) &&\quad\text{ on }\partial \Dsf.
\end{cases2}
Notice that in comparison to the linear setting studied in the previous section we now have an additional term on the right hand side, namely, $R(x-x_0)\partial_u\rho_\Om(x,u_\Om)$, which accounts for the fact that the equation \eqref{eq:R} does not have a lower order term. Then it is readily checked that
\begin{equation}
    U_0(x) := |\omega|^{-1}(R(x-x_0) + v(x)), \quad \text{ a.e. }  x\in \Dsf,
\end{equation}
solves the equation \eqref{eq:limit_U0_semilinear_asympt1}.

We now show that $U_0$ can be obtained from the asymptotics of $u_\Om$. We let $\omega$, $x_0$ and $\Omega_\eps = \Omega_\eps(x_0,\omega)$ be as in the previous section. Denote again the solution of \eqref{Eq:MainSemiLinear2} $ u_{\Om_\eps}$ by $u_\eps$ and $u_0:= u_\Om$. Following the formal asymptotic expansion of \cite{a_IGNAROSOSZ_2009a} we have the following expansion of $u_{\Om_\eps}$: 
    \begin{equation}\label{eq:semi_point_wise_d2}
    u_\eps(x) = u_0(x) + \eps^2( K(T_\eps^{-1}(x)) + v(x) + \ln(\eps)b) +o(\eps^2) \quad \text{ for } d=2,
\end{equation}
with $b:= - \mathrm{sgn}_\Om(\{x_0\}) g_{x_0} |\omega|(2\pi)^{-1}$ and  for $d=3$
\begin{equation}\label{eq:semi_point_wise_d3}
    u_\eps(x) = u_0(x) + \eps^2( K(T_\eps^{-1}(x)) + \eps v(x)) + o(\eps^3) \quad \text{ for } d=3.
\end{equation}
Here $K$, given by
\begin{equation}\label{eq:K_semilinear_explicit}
    K(x) = \mathrm{sgn}_\Om(\{x_0\}) g_{x_0} \int_\omega E(x-y)\;dy,
\end{equation}
solves for $d\in \{2,3\}$ the equation
\begin{equation}\label{eq:K_semilinear}
    -\Delta K = \mathrm{sgn}_\Om(\{x_0\}) g_{x_0} \chi_{\omega} \quad \text{ in } \VR^d.
\end{equation}

\begin{theorem}\label{thm:main_rates_semilinear}
For $x_0\in \Dsf\setminus \partial \Omega$ and $\omega\subset \VR^d$ with $0\in \omega$ we use the definition of $\Omega_\eps(x_0,\omega)$ of \eqref{eq:omega_eps} and set $u_\eps := u_{\Omega_\eps(x_0,\omega)}$ and $u_0:= u_\Omega$ and
\begin{equation}
    U_\eps := \frac{u_\eps - u_0}{|\omega_\eps|}, \quad \eps >0. 
\end{equation}
Let $K$ and $v$ be defined by \eqref{eq:K_semilinear_explicit} and \eqref{eq:vk_semilinear1}, respectively. Then we have the following results.
    From the asymptotic expansion \eqref{eq:semi_point_wise_d2} and \eqref{eq:semi_point_wise_d3}, we conclude that 
    the limit \eqref{eq:limit_semilinear} exists. In fact we have
 \begin{equation}
     U_0(x) =  |\omega|^{-1}(R(x-x_0) + v(x)), \quad \text{ for a.e. } x\in \Dsf,
\end{equation}
and for a.e. $x\in \Dsf$, we have with $b:= - \mathrm{sgn}_\Om(\{x_0\})|\omega|g_{x_0}(2\pi)^{-1}$:
\begin{equation}\label{eq:pointwise_limit_semilinear_dim2}
    U_0(x) = \lim_{\eps\searrow 0}\frac{1}{|\omega_\eps|} \eps^2(K(T_\eps^{-1}(x)) + v(x) + \ln(\eps)b) \quad \text{ for } d=2,
\end{equation}
and 
\begin{equation}\label{eq:pointwise_limit_semilinear_dim3}
    U_0(x) = \lim_{\eps\searrow 0}\frac{1}{|\omega_\eps|}  \eps^2 (K(T_\eps^{-1}(x)) + \eps v(x)) \quad \text{ for } d= 3.
\end{equation}
\end{theorem}
The proof of this theorem follows the lines of the proof of item (i) of Theorem~\ref{thm:main_rates_rhs}.

\section{Main results for operator point perturbation in linear transmission problems}\label{Se:Op}
In this section we show how our type of analysis carries on to transmission problems, which is also referred to as "perturbation of the operator". The analysis of this type of perturbation is more difficult than perturbations of lower order terms and typically involves so-called ``polarisation matrices''; \cite{a_AM_2006a,b_NOSO_2013a,b_AMKA_2007a}. We will see that the topological state derivative for point perturbations of the operator exists, but exhibits a very low regularity.

\subsection{Topological state derivative for the transmission problem}
\paragraph{Analytic set-up}
Throughout this section we fix a set $\Om\in \Ca(\Dsf)$ and assume $\Omega\Subset \Dsf$ is smooth. For two fixed parameters $\beta_1\,, \beta_2>0$ and every $\Om\in \Ca(\Dsf)$, we define 
\[ 
\beta_\Om:=\beta_1\chi_\Om+\beta_2\chi_{\Dsf\backslash\overline\Om}.
\]
Now for $f\in L^p(\Dsf)$ with $p>d$, let $u_\Om$ be the unique weak solution in $W^{1,p}(\Dsf)$ of the equation 
\begin{cases2}{Eq:MainTransmission}
-\Div(\beta_\Om \nabla u_\Om)&=f &&\quad\text{ in }\Dsf, 
\\ u_\Om&=0&&\quad\text{ on }\partial \Dsf.
\end{cases2}
We note that the strong form of this equation reads: denoting by $u^+:=u_{|\Omega}$ and $u^-:= u_{|\Dsf\setminus\overline\Omega}$, we have
\begin{cases2}{strongtrans}
    -\beta_1 \Delta u^+ & =  f && \text{ in } \Omega, \\
    -\beta_2 \Delta u^- & =  f && \text{ in } \Dsf\setminus\overline\Omega, \\
    \beta_1\partial_\nu u^+& = \beta_2\partial_\nu u^- && \text{ on } \partial \Omega, \\
    u^+ & = u^- && \text{ on } \partial \Omega, \\
    u^- & = 0 && \text{ on } \partial \Dsf.
\end{cases2}

\paragraph{Topological perturbation under consideration}
In this section, we consider only point perturbations $\Omega_\eps(x_0,\omega)$ at points $x_0\in \Dsf\setminus\partial \Omega$ of the set $\Omega$ defined by
\begin{align}\label{eq:omega_epss}
    \Omega_\eps & := \Omega_\eps(x_0,\omega) :=  \begin{cases}
     \Omega\cup \omega_\eps(x_0) &   \text{ for }x_0\in \Dsf\setminus \overline\Omega, \\
     \Omega\setminus \overline{\omega_\eps(x_0)} &  \text{ for } x_0 \in \Omega,
\end{cases}  
\end{align}
and $\omega_\eps(x_0):= x_0 + \eps \omega$ with $\omega\subset \VR^d$ being a simply connected domain with $0\in \omega$. We note that 
the special case $\omega=B_1(0)$ would correspond to the dilatation  $\Omega(E_\eps)$ for $E=\{x_0\}$ considered in the previous sections. Following   \cite{a_AM_2006a,a_ST_2020a} we decided to treat more general perturbations.

As we will see, the right hand side of the limit equation $U_0$ for the transmission problem will involve the divergence of a measure. As this is the only part which can not be covered by the techniques thus far used in studying semilinear models, we devote a paragraph to some basic definitions.

\paragraph{Very weak solutions for elliptic equations with divergence-of-measure right hand side}
We consider, for a given $\zeta\in \mathbb R^d$, the Dirac measure $\mu:= \zeta \delta_{x_0} \in \mathcal M(\Dsf)^d$ concentrated at $x_0\in \Dsf\setminus\partial\Omega$ , the equation 
\begin{cases2}{}
-\Div(\beta_\Om \nabla \varphi_\mu)&=\Div(\mu)&&\quad\text{ in }\Dsf, 
\\ \varphi_\mu&=0&&\quad\text{ on }\partial \Dsf.
\end{cases2}
The weak formulation would read
\begin{equation}\label{eq:weak1}
    \int_\Dsf \beta_\Om \nabla \varphi_\mu\cdot \nabla \varphi \;dx =  \zeta\cdot\nabla \varphi(x_0)\quad \text{ for all } \varphi \in H^1_0(\Dsf),
\end{equation}
which is obviously not well-defined. At this point, let us observe that by interior regularity $\nabla u_\Om$ is continuous in a neighborhood of the perturbation point $x_0\in \Dsf\setminus\partial\Omega$. Furthermore, we also need to ensure that $\nabla \varphi$ is continuous as well. We thus resort to a notion of weak solution reminiscent of the one introduced in \cite{Littman}: for every $p>d$, for $v\in L^p(\Dsf)$, let $\varphi_v \in W^{2,p}(\Om\cup(\Dsf\setminus\overline{\Om}))\cap H^1_0(\Dsf)$ be the unique solution of 
\begin{cases2}{def:varphiv}
-\Div(\beta_\Omega\nabla \varphi_v)&=v&&\quad\text{ in }\Dsf, 
\\ \varphi_v&=0&&\quad\text{ on }\partial \Dsf.
\end{cases2}
The well-posedness of this equation follows from arguments similar to the ones used in the proof of Proposition \ref{Pr:Stampacchia}.
By elliptic regularity, for every $p>d$, we have, for two constants $C_p\,, C_p'$
\begin{equation}\label{eq:embedding}
    \|\varphi_v\|_{C^1(\Om\cup (\Dsf\setminus \Om))} \le C_p\|\varphi_v\|_{W^{2,p}(\Om\cup (\Dsf\setminus\Om))} \le C_p'\|v\|_{L^p(\Dsf)}, \quad \text{ for all } v\in L^p(\Dsf).
\end{equation}
Now choosing $\varphi_v$ in \eqref{eq:weak1} and  integrating by parts we obtain
\begin{equation}
 \int_\Dsf u_\Omega v\; dx = \int_\Dsf \beta_\Omega\nabla u_\Omega\cdot \nabla \varphi_v\;dx =\int_\Dsf \langle \nabla u_\Om,\nabla \varphi_v\rangle d\mu = \zeta \cdot \nabla \varphi_v(x_0) \quad \text{ for all } v\in L^p(\Dsf). 
\end{equation}
We note that the interface terms vanish in view of the choice of the test function $\varphi_v$. This leads to the following definition:
\begin{definition}\label{De:VeryWeak}
Let $q\in [1,\frac{d}{d-1})$ and $q^\prime$ be its conjugate Lebesgue exponent. We say that $\phi_\mu$ is a very weak $W^{1,q}_0(\Dsf)$-solution of \eqref{eq:weak1} if 
\begin{equation}
    \int_\Dsf \phi_\mu v = \nabla \varphi_v(x_0)\cdot \zeta \quad \text{ for all } v\in L^{q^\prime}(\Dsf),
\end{equation}
where $\varphi_v\in W^{1,q^\prime}_0(\Dsf)$ solves
\begin{cases2}{defvarphiv}
-\Div(\beta_\Omega\nabla \varphi_v)&=v&&\quad\text{ in }\Dsf, 
\\ \varphi_v&=0&&\quad\text{ on }\partial \Dsf.
\end{cases2}
It is important to note that the $\Omega$ dependence is now transferred to the definition of $\varphi_v$. 
\end{definition}

The main proposition is the following:
\begin{proposition}\label{Pr:LinearisedTransmission}
    Let $x_0\in \Dsf\setminus\partial \Omega$. For $\zeta\in \VR^d$ let $\mu:= \zeta \delta_{x_0}$. 
The equation 
\begin{cases2}{eq:weak2}
-\Div(\beta_\Om\nabla \phi_{\mu})&=\Div(\mu )&&\quad\text{ in }\Dsf,
\\ \phi_{\mu}&=0&&\quad\text{ on }\partial \Dsf,
\end{cases2}
has a unique very weak solution in the sense of Definition~\ref{De:VeryWeak}. Furthermore, for every $q\in [1,\frac{d}{d-1})$, there exists a constant $C_q>0$ such that
\[ 
\Vert \phi_{\mu}\Vert_{L^q(\Dsf)}\leq C_q\Vert \mu\Vert_{\mathcal M(\Dsf)}.
\]
\end{proposition}

\paragraph{Expression of the topological state derivative}

Our main theorem is the following:
\begin{theorem}\label{Th:MainTransmission}
    Let $\Omega\in \Ca(\Dsf)$ with $\Omega\Subset \Dsf$. Let $x_0,\omega$ and $\Omega_\eps(x_0,\omega)$ be as in \eqref{eq:omega_epss} and denote by $u_\Om$ the unique weak solution to \eqref{Eq:MainTransmission}.   We have that $U_\eps=\frac{u_{\Omega_\eps(x_0,\omega)}-u_\Omega}{|\omega_\eps|}$ is bounded in $L^q(\Dsf)$ for $q\in [1,\frac{d}{d-1})$ and we have for a constant $C=C_{q,d}>0$, depending on $q$ and $d$:
  \begin{equation}\label{eq:trans_estimate}
        \|U_\eps - U_0\|_{L^q(\Dsf)} \le C(\eps^{\frac{d-q(d-1)}{q}} + \|K_\eps - K\|_{L^1(\omega)})
    \end{equation}
and thus in particular $U_\eps \to U_0$ strongly in $L^q(\Dsf)$ as $\eps\searrow 0$.  The limit $U_0$ solves:
    \begin{equation}\label{eq:limit_U_transmission}
        \int_\Dsf U_0 v \;dx = \mathrm{sgn}_\Om(\{x_0\}) (\beta_2-\beta_1)\left(\frac{1}{|\omega|}\int_\omega \nabla K + \nabla u_0(x_0)\; dx\right)\cdot \nabla \varphi_v(x_0)  \quad \text{ for all } v\in L^{q^\prime}(\Dsf),
    \end{equation}
    where $\varphi_v$ is the solution to \eqref{def:varphiv}.  Here, $K$ belongs to the Beppo-Levi space $\dot{\text{BL}}(\VR^d)$\footnote{The Beppo-Space $\dot{\text{BL}}(\VR^d)$ is defined as the quotient space $\{\varphi \in H^1_{loc}(\VR^d): \nabla \varphi \in L^2(\VR^d)^d\}/\VR$, where $/\VR$ means we quotient out constants. This space is equipped with the norm 
    $\|\varphi\|_{\dot{\text{BL}}(\VR^d)}:= \|\nabla \varphi\|_{L^2(\VR^d)^d}$; see \cite{a_ST_2020a,a_DELI_1955a}.} and is the unique solution to:
\begin{equation}\label{eq:equation_K2}
    \int_{\VR^d} \beta_\omega \nabla K\cdot \nabla \varphi \;dx = \mathrm{sgn}_\Om(\{x_0\})(\beta_2-\beta_1)\int_\omega \nabla u_0(x_0)\cdot \nabla \varphi \;dx,
\end{equation}
for all $\varphi \in \dot{\text{BL}}(\VR^d)$. 
\end{theorem}

\begin{remark}[Polarisation matrix]\label{rem:polarisation}
We note that $K=K[\nabla u_0(x_0)]$ actually depends linearly on the vector $\nabla u_0(x_0)$ through the equation \eqref{eq:equation_K2}. Consequently, the map
\begin{equation}
    \nabla u_0(x_0) \mapsto \frac{1}{|\omega|} \int_\omega \nabla K \;dx: \VR^d\to \VR^d,
\end{equation}
is also linear and thus there is a so-called polarisation matrix $A_\omega\in \VR^d$; see \cite{b_AMKA_2007a,a_AM_2006a}, which depends on $\beta_1,\beta_2$ and $\omega$, such that
\begin{equation}
    A_\omega \nabla u_0(x_0) = \frac{1}{|\omega|} \int_\omega \nabla K \;dx.
\end{equation}
It follows that the equation \eqref{eq:limit_U_transmission} is equivalent to 
    \begin{equation}\label{eq:limit_U_transmission2}
        \int_\Dsf U_0 v \;dx = \mathrm{sgn}_\Om(\{x_0\}) (\beta_2-\beta_1) (A_\omega+I_d)\nabla u_0(x_0)\cdot \nabla\varphi_v(x_0)   \quad \text{ for all } v\in L^p(\Dsf),
    \end{equation}
    where $I_d\in \VR^{d\times d}$ denotes the identity matrix.\newline
Considering the special case $\omega=B_1(0)$, one readily checks that the solution $K$ of equation \eqref{eq:equation_K2} is given by
\begin{equation}
K(x)=\begin{cases}C_\beta \nabla u_0(x_0)\cdot x,\quad &\text{ for } x\in\omega,\\
C_\beta \nabla u_0(x_0)\cdot\frac{x}{|x|^d},\quad&\text{ for }x\in \VR^d\setminus\overline\omega,\end{cases}
\end{equation}
where $C_\beta:=\mathrm{sgn}_\Om(\{x_0\})(\beta_2-\beta_1)(\beta_1-\beta_2-d\beta_2)^{-1}$. Thus, for the unit ball inclusion, the polarisation matrix is given by $A_\omega=C_\beta I_d$.
\end{remark}

\subsection{Asymptotic expansion of $u_\Omega$ and the relation to the topological state derivative}
We start this section by giving some results regarding the asymptotic expansion of $u_{\Omega_\eps(x_0,\omega)}$. Note that these are derived using compound asymptotics; see \cite{a_BAST_2021a,b_MANAPL_2012a,b_NOSO_2013a}.

\paragraph{Asymptotic analysis of $u_\Omega$}
Let $x_0\in \Dsf\setminus\partial \Omega$, $\omega\subset\VR^d$ a simply connected and bounded domain with $0\in \VR^d$ and set $u_\eps:= u_{\Omega_\eps(x_0,\omega)}$, where 
    $\Omega_\eps(x_0,\omega)$ is as defined in \eqref{eq:omega_eps}. For $\eps >0$ we introduce
\begin{equation}
    K_\eps := \frac{(u_\eps - u_0)\circ T_\eps}{\eps}. 
\end{equation}
It is a classical result that the limit of $K_\eps$ is in fact the unique solution to: find $K\in \dot{\text{BL}}(\VR^d)$ such that
\begin{equation}\label{eq:equation_K}
    \int_{\VR^d} \beta_\omega \nabla K\cdot \nabla \varphi \;dx =\mathrm{sgn}_\Om(\{x_0\})(\beta_2-\beta_1)\int_\omega \nabla u_0(x_0)\cdot \nabla \varphi \;dx,
\end{equation}
for all $\varphi \in \dot{\text{BL}}(\VR^d)$. The function $K$ admits the asymptotics
\begin{equation}
    K(x) = R(x) + O\left(\frac{1}{|x|^d}\right).
\end{equation}
To state the final asymptotic expansion we need to introduce the regular boundary corrector $v$ compensating the error introduced by $K$ on $\partial \Dsf$.

 The corrector $v\in H^1(\Dsf)$ is defined as the unique solution to $v(x)=-R(x-x_0)$ on $\partial \Dsf$ and
\begin{equation}\label{eq:corr_v_transmission}
    \int_\Dsf \beta_\Omega \nabla v \cdot \nabla \varphi \;dx = \mathrm{sgn}_\Om(\{x_0\})(\beta_2-\beta_1)\int_{\partial \Omega} \partial_\nu R(x-x_0) \varphi(x) \;dx  \quad \text{ for all } \varphi \in H^1_0(\Dsf).
\end{equation}

The following lemma states the main result regarding the first order asymptotic expansion. 
We closely follow the arguments of \cite[Theorem 3.15]{a_BAST_2021a}, but since we require estimates in $L^q$, we provide the main 
steps of the proof in the appendix.
\begin{lemma}\label{lem:convergence_Keps_transmission}
    For $q\in (1,\frac{d}{d-1})$ there is a constant $C=C_{q,d}>0$, which depends on $q$ and $d$, such that for all $\eps >0$ small:
\begin{equation}\label{eq:asymp_exp_trans}
    \|\eps(K_\eps - K - \eps^{d-1} v\circ T_\eps)\|_{L^q(\Dsf_\eps)} +  \|\nabla(K_\eps - K - \eps^{d-1}v\circ T_\eps)\|_{L^q(\Dsf_\eps)^d} \le C\eps,
\end{equation}
which, by considering the scaling of the norms and $\Dsf_\eps = T^{-1}_\eps(\Dsf)$, is equivalent to
\begin{equation}\label{eq:asymp_exp_trans_fixed}
    \|K_\eps\circ T_\eps^{-1} - K\circ T_\eps^{-1} - \eps^{d-1} v\|_{W^{1,q}(\Dsf)} \le C\eps^{\frac{d}{q}}.
\end{equation}
\end{lemma}

\paragraph{Relation between asymptotic expansion and topological state derivative}
We now want to make the connection between the topological expansion and the topological state derivative. 
For this we note that the estimate \eqref{eq:asymp_exp_trans} reads on the fixed domain $\Dsf$:
\begin{equation}
    u_\eps = u_0 + \eps K(T_\eps^{-1}(x)) + \eps^d v(x) + \text{ higher order terms}. 
\end{equation}
To see the relation between the topological state derivative $U_0$ of Theorem~\ref{Th:MainTransmission} and the asymptotic expansion (Lemma \eqref{lem:convergence_Keps_transmission}), we first note that 
the equation \eqref{eq:equation_K} can be written as follows
\begin{equation}\label{eq:K_transmission}
       \int_{\VR^d} \nabla K\cdot \nabla \varphi \;dx =\mathrm{sgn}_\Om(\{x_0\})\frac{\beta_2-\beta_1}{\beta(x_0)}\left(\int_\omega \nabla u_0(x_0)\cdot \nabla \varphi \;dx + \int_\omega \nabla K \cdot \nabla \varphi \;dx  \right),
\end{equation} 
where  $\beta(\cdot)$ is a piecewise constant function defined by 
\[
    \beta(x_0)=\begin{cases} 
        \beta_2 & \text{ for } x_0\in \Dsf\setminus\overline\Omega, \\
        \beta_1 & \text{ for } x_0 \in \Omega.

    \end{cases}
\]
Therefore, with $E(\cdot)$ denoting the fundamental solution of $-\Delta$ on $\VR^d$, it is a classical result that $K$ can be expressed as follows
\begin{equation}
    K(x) =  \mathrm{sgn}_\Om(\{x_0\})\frac{\beta_2-\beta_1}{\beta(x_0)}\left(\int_\omega \nabla u_0(x_0)\cdot \nabla E(x-y) \;dy + \int_\omega \nabla K(y) \cdot \nabla E(x-y) \;dy  \right).
\end{equation}
Therefore, performing a Taylor expansion, we see that the first asymptotic term of $K(x)$ as $|x|\to\infty$ can be written as
\begin{equation}
    R(x) =  \underbrace{\mathrm{sgn}_\Om(\{x_0\})\frac{\beta_2-\beta_1}{\beta(x_0)}\left(\int_\omega \nabla u_0(x_0) \;dy + \int_\omega \nabla K(y) \;dy  \right)}_{\xi_{x_0}}\cdot \nabla E(x),
\end{equation}
where $\xi_{x_0}$ is a vector depending on $\nabla u_0(x_0)$.  Notice that $\xi_{x_0}$ can also be expressed through the polarisation matrix $A_\omega$ of Remark~\ref{rem:polarisation} as follows
\begin{equation}
    \xi_{x_0} = \mathrm{sgn}_\Om(\{x_0\})\frac{\beta_2-\beta_1}{\beta(x_0)}|\omega| (A_\omega+I_d)\nabla u_0(x_0)\cdot \nabla E(x)
\end{equation}
making the dependence on $x_0$ more explicit.
Now we note that the function $U_0(x):= |\omega|^{-1}(\xi_{x_0}\cdot \nabla E(x-x_0) + v(x))$ solves in a very weak sense:
\begin{align}
    \int_\Dsf \beta_\Om \nabla U_0 \cdot \nabla \varphi \;dx =& \mathrm{sgn}_\Om(\{x_0\}) (\beta_2-\beta_1)\left(\frac{1}{|\omega|}\int_{\omega} (\nabla K  + \nabla u_0(x_0))\cdot \nabla \varphi(x_0) \;dx\right),
\end{align}
for all $\varphi\in C^1_c(\Dsf)$. In fact, $\xi_{x_0}\cdot \nabla E(x-x_0)\in L^q(\Dsf)$ for $q\in [1,\frac{d}{d-1})$ and thus one readily verifies that $U_0(x)$ is indeed a very weak solution as defined in Definition~\ref{eq:weak2}.

With this we can state our next main result linking the asymptotic expansion and the topological state derivative:

\begin{theorem}\label{thm:mainThmTransmissionAsymptotic}
For $x_0\in \Dsf\setminus \partial \Omega$ and $\omega\subset \VR^d$ with $0\in \omega$ we use the definition of $\Omega_\eps(x_0,\omega)$ of \eqref{eq:omega_eps} and set $u_\eps := u_{\Omega_\eps(x_0,\omega)}$ and $u_0:= u_\Omega$ and
\begin{equation}
    U_\eps := \frac{u_\eps - u_0}{|\omega_\eps|}, \quad \eps >0. 
\end{equation}
Let $K$ and $v$ be defined by \eqref{eq:equation_K} and \eqref{eq:corr_v_transmission}, respectively. Then we have the following results.

\begin{itemize}
    \item[(i)] The asymptotic expansion \eqref{eq:asymp_exp_trans} yields that there is a constant $C=C_{q,d}>0$, which depends on $q$ and $d$, such that for all $q\in (1,\frac{d}{d-1})$
        \begin{equation}\label{eq:estimate_trans_strong}
        \|U_\eps-|\omega|^{-1}(R(x-x_0)+v)\|_{L^q(\Dsf)}\le C \eps^{\frac{d-q(d-1)}{q}}\quad \text{ for all } \eps >0. 
    \end{equation}
\item[(ii)] For a.e. $x\in \Dsf$:
\begin{equation}\label{eq:limit_U0}
    U_0(x) = \lim_{\eps\searrow0}\frac{1}{|\omega_\eps|}( \eps K(T^{-1}_\eps(x)) + \eps^d v(x)),
\end{equation}
and thus in particular
\begin{equation}\label{eq:limit_U02}
    U_0(x) = |\omega|^{-1}(\xi_{x_0}\cdot \nabla E(x-x_0) + v(x)),
\end{equation}
with
\begin{equation}
    \xi_{x_0} := \mathrm{sgn}_\Om(\{x_0\})\frac{\beta_2-\beta_1}{\beta(x_0)}\left(\int_\omega\nabla u_0(x_0)\; dy + \int_\omega \nabla K(y) \;dy  \right).
\end{equation}
\end{itemize}
\end{theorem}
Note that for $q\in (1,\frac{d}{d-1})$, the exponent ${\frac{d-q(d-1)}{q}}$ is indeed positive.

\begin{remark}
In (i) we only claim the convergence rate ${\frac{d-q(d-1)}{q}}$ for $q>1$ while $q=1$, which would correspond to the convergence rate $\eps$, is excluded. Obviously we also obtain strong convergence in $L^1$ via H\"older's inequality, but the estimate \eqref{eq:estimate_trans_strong} only holds for $q\in (1,\frac{d}{d-1})$.
\end{remark}

\section{Topological derivatives of shape functions via topological state derivative} \label{Se:Adjoint}

\subsection{Topological differentiability of shape functionals}
\paragraph{Semilinear problem}
We denote by  $S(\Omega)=u_\Omega$ the solution operator of the semilinear equation \eqref{Eq:MainSemiLinear}. 
We now discuss the differentiability of $\Omega\mapsto J(\Omega) = \F(S(\Omega))$ for a cost functional $\F:W^{1,q}(\Dsf)\to \VR$, $q\in [1,\frac{d}{d-1})$. In fact, under sufficient differentiability assumptions on $J$ and if $\F'(u_\Omega):W^{1,q^\prime}(\Dsf) \to \VR$ is well-defined, we can show that
\[
DJ(\Omega)(E) = \lim_{\eps\searrow 0} \frac{\F(u_{\Om(E_\eps)}) - \F(u_\Om)}{|E_\eps|} = \F'(S(\Omega))(S'(\Omega)(E)),
\]
where $E\subset \Om$ or $E\subset \Dsf\backslash \Om$ is either a point, $d-1$-dimensional Lipschitz surface or $1<k<d-1$ dimensional compact and smooth submanifold and $\Omega(E_\eps)$ is defined in \eqref{De:Main}. Our goal is now to compute the limit $\eps\searrow 0$ of 
\[
    U_\eps:=\frac{u_\eps-u_0}{|E_\eps|}, \quad \eps>0,
\]
where $u_\Om$ is the solution to the semilinear problem \eqref{Eq:MainSemiLinear}. To simplify notation, we define again for every $\eps>0$ the functions $u_\eps := u_{\Om(E_\eps)}$, $u_0:=u_\Om$. Recall by Theorem \ref{Th:MainSemiLinear}, $U_\eps\rightarrow U_{0}=S'(\Omega)(E)$ in $L^q(\Dsf)$, $q\in[1,\frac{d}{d-1})$.  Moreover we have according to \cite[Lemma~4.4]{a_ST_2020a} that $u_\eps\rightarrow u_0=u_\Om$ in $H^1_0(\Dsf)$ and thus via the Sobolev inequality $u_\eps \to u_0$ in $L^p(\Dsf)$ for $1\le p<\frac{2d}{d-2}$ for $d\ge 3$ and $1\le p<\infty$ for $d=2$.

\begin{example}[$L^2$ tracking-type]\label{ex:tracking_l2}
A classical example of a cost functional $\F(\cdot)$ is of tracking-type, that is:
\begin{equation}
    \F(u):= \int_\Dsf (u -u_{\text{ref}})^2 \;dx, \quad u_{\text{ref}}\in L^2(\Dsf). 
\end{equation}
 Consequently,  the topological derivative of $J(\Omega):= \F(u_\Omega)$ is given by 
    \begin{equation}
    DJ(\Omega)(E) = \lim_{\eps\searrow 0}\int_\Dsf U_{\eps}(u_\eps+u_0-2u_{\text{ref}})\;dx = 2\int_\Dsf S'(\Omega)(E)(u_\Omega - u_{\text{ref}})\;dx,
\end{equation}
\end{example}

\begin{example}[$L^r$ tracking]\label{ex:tracking_lr}
More generally we can differentiate for $r>2$
\begin{equation}
    \F(u):= \int_\Dsf (u-u_{\text{ref}})^r\;dx, \quad  u_{\text{ref}}\in L^\infty(\Dsf)k.
\end{equation}
Let again $J(\Omega):= \F(u_\Omega)$. Then we have
\begin{equation}
    DJ(\Omega)(E) = r\int_\Dsf S'(\Omega)(E)(u_\Omega -u_{\text{ref}})^{r-1}\;dx.
\end{equation}

\end{example}

Another classical example is the gradient tracking type cost functional.
\begin{example}[$L^2$ gradient tracking]\label{ex:tracking_H1}
For $u_{\text{ref}}\in W^{1,q^\prime}(\Dsf)$, consider the gradient-tracking functional
\begin{equation}
    \F(u):= \int_\Dsf |\nabla u -\nabla u_{\text{ref}}|^2 \;dx.
\end{equation}
A similar computation to the previous one shows
\begin{equation}
DJ(\Omega)(E)=2\int_\Dsf \nabla \left(S'(\Omega)(E)\right)\cdot\nabla (u_\Omega - u_{\text{ref}})\;dx.
\end{equation}
\end{example}

\paragraph{Transmission problem}
Note that due to the weaker convergence result for the transmission problem, the computation of the topological derivative can be more involved for certain cost functionals. We give the following examples. 

Let $\Omega\in \Ca(\Dsf)$. For $x_0\in \Dsf\setminus \partial \Omega$ and a simply connected $\omega\subset \VR^d$ with $0\in \omega$ we use the definition of $\Omega_\eps(x_0,\omega)$ given in \eqref{eq:omega_eps}. We denote by $u_\Om$ the solution to \eqref{Eq:MainTransmission}.  We set $u_\eps := u_{\Omega_\eps(x_0,\omega)}$ and $u_0:= u_\Omega$ and
\begin{equation}
    U_\eps := \frac{u_\eps - u_0}{|\omega_\eps|}, \quad \eps >0. 
\end{equation}
Recall that according to Theorem~\ref{Th:MainTransmission} we have $U_\eps \to U_0$ as $\eps\searrow 0$ in $L^q(\Dsf)$ for $q\in [1,\frac{d}{d-1})$.  We also have according to \cite[Lemma~4.4]{a_ST_2020a} that $u_\eps\to u_0$ in $H^1_0(\Dsf)$, which implies by the Sobolev inequality $u_\eps \to u_0$ in $L^p(\Dsf)$ for $1\le p<\frac{2d}{d-2}$ for $d\ge 3$ and $1\le p<\infty$ for $p=2$. With the definition of $\Omega_\eps(x_0,\omega)$ given in \eqref{eq:omega_epss}, we define the topological derivative for $x_0\in \Dsf\setminus\partial \Omega$:
\begin{equation}\label{def:topo_derivative_trans}
    DJ(\Omega)(x_0,\omega) = \lim_{\eps\searrow0} \frac{J(\Omega_\eps(x_0,\omega)) - J(\Omega)}{|\omega_\eps|}.
\end{equation}
Notice that in case $\omega=B_1(0)$ we have with  $E=\{x_0\}$ that $DJ(\Omega)(E) = DJ(\Omega)(x_0,\omega)$, so the previous definition of the topological derivative given in \eqref{def:topo_derivative} coincides with \eqref{def:topo_derivative_trans}. In contrast to the semilinear problem the topological derivative for the transmission problem actually depends on the shape $\omega$ of the inclusion.

\begin{example}[$L_2$ tracking-type]\label{ex:transmission}
For $u_{\text{ref}}\in L^{q^\prime}(\Dsf)$ consider the tracking type cost functional
\begin{equation}
    J(u_\Omega):= \int_\Dsf (u_\Omega -u_{\text{ref}})^2 \;dx.
\end{equation}
The topological derivative is given by
\begin{equation}
    DJ(\Omega)(x_0,\omega) = \lim_{\eps\searrow 0}\int_\Dsf U_{\eps}(u_\eps+u_0-2u_{\text{ref}})\;dx = 2\int_\Dsf U_0(u_\Omega - u_{\text{ref}})\;dx.
\end{equation}
 In view of the regularity of $U_0$ and $u_\Om,u_{\text{ref}}$ the last integral is indeed well-defined.
\end{example}
As a second example we consider the energy minimisation,  where the topological derivative cannot be directly computed via the topological state derivative.
\begin{example}[$L_2$ gradient tracking]\label{ex:transmission_energy}
Let
\begin{equation}
    J(u_\Omega):= \int_\Dsf \beta_\Omega |\nabla u_\Omega|^2 \;dx.
\end{equation}
As before, we compute
\begin{equation}\label{eq:limit_transmission}
    DJ(\Omega)(x_0,\omega) = \lim_{\eps\searrow 0}\int_\Dsf \beta_\Omega \nabla U_\eps\cdot (\nabla u_\eps+\nabla u_0)\;dx+\frac{1}{|\omega_\eps|}\int_{\omega_\eps}(\beta_1-\beta_2)|\nabla u_\eps|^2\;dx.
\end{equation}
Since the convergence $\nabla U_\eps \to \nabla U_0$ as $\eps\searrow 0$ does not hold in $L^p(\Dsf)^d$ for $p\ge 2$, we cannot pass to the limit. However, it can be shown using a Lagrangian framework as in \cite{a_ST_2020a,a_BAST_2021a} that, in fact, the first derivative \eqref{eq:limit_transmission} exists. This means that we cannot use the chain rule and is clearly a limitation of the use of the topological state derivative. 
\end{example}

\subsection{Expression of topological derivative of functionals with adjoint equation}

\paragraph{Semilinear problem}

We now express the derivative of $J(\Omega) = G(u_\Omega)$, where $G:W^{1,q^\prime}(\Dsf)\to\VR$, with $q'>d$ (or equivalently $1\le q<\frac{d}{d-1}$), is given, for the semilinear problem in terms of the adjoint equation. We introduce the adjoint state $p_\Om\in W^{1,q^\prime}_0(\Dsf)$, that is, the unique solution of

\begin{cases2}{eq:adjoint_semi}
-\Delta p_\Om+\partial_u\rho_\Om(x,u_\Om)p_\Om &= -\F'(u_\Om) &&\quad\text{ in }\Dsf, 
\\ p_\Om=&0&&\quad\text{ on }\partial \Dsf.
\end{cases2}

\begin{theorem}
Assume that $\F(\cdot)$ is differentiable, such that with $J(\Omega):= \F(u_\Omega)$ it holds
\begin{equation}
    DJ(\Omega)(E) = \F'(S(\Omega))(S'(\Omega)(E)). 
\end{equation}
Then we have 
\begin{equation}
    DJ(\Om)(E)  = -\mathrm{sgn}_\Om(E) \mu_{E}(W_\Om) \quad \text{ with } \quad W_\Om  = \left\{(g_2(u_\Om)-g_1(u_\Om))+(f_1-f_2)\right\}p_\Om.
\end{equation}
\begin{itemize} 
    \item For $E=\{x_0\}$ and $x_0\in \Dsf\setminus\partial \Om$ we have $\mu_{E}(W_\Om) = W_\Om(x_0)$.
\item For $E=\Gamma$ for $\Gamma\subset \Dsf\setminus \partial \Om$, where $\Gamma$ is a smooth hypersurface  of $\VR^d$, we have
\begin{equation}
    \mu_{E}(W_\Om) = \frac{1}{\text{Per}(\Gamma)}\int_{\Gamma} W_\Om\; d\Ch^{d-1}.
\end{equation}
\item For $E=M$ for $M\subset \Dsf\setminus\partial \Omega$, where $M$ is a smooth $k$-submanifold, $1< k <d-1$ of $\VR^d$
    without boundary, we have
    \begin{equation}
        \mu_{E}(W_\Om) = \frac{1}{\Ch^k(M)}\int_{M} W_\Omega \; d\Ch^k.
    \end{equation}
\end{itemize}
\end{theorem}
\begin{proof}
Recalling that by definition $S(\Om)=u_\Om$, we have
\begin{align}
    DJ(\Om)(E) = \F'(S(\Om))(S'(\Om)(E))  & \stackrel{\eqref{eq:adjoint_semi}}{=} - \int_\Dsf \nabla p_\Om \cdot \nabla S'(\Om)(E) + \partial_u\rho_\Om(x,u_\Om)p_\Om S'(\Om)(E) \;dx  \\
                                        & \stackrel{\eqref{Eq:UeSemiLinear}}{=}  -\mathrm{sgn}_\Om(E)\mu_{E}(W_\Omega).
\end{align}
This concludes the proof.
\end{proof}

\begin{example}
    Consider $J(\Omega)=\F(u_\Omega)$ and again the tracking-type cost functional of Example~\ref{ex:tracking_l2}, namely, 
    \[
        \F(u) = \int_\Dsf(u - u_{\text{ref}})^2\;dx.
    \]
    The adjoint state is the (unique) solution $p_\Om \in W^{1,q^\prime}_0(\Omega)$, $q\in [1,\frac{d}{d-1})$, of 
    \begin{cases2}
        -\Delta p_\Om+\partial_u\rho_\Om(x,u_\Om)p_\Om & = -2(u_\Om - u_{\text{ref}})  && \quad \text{ in } \Dsf, \\
        p_\Om & = 0 && \quad \text{ on }\partial \Dsf,
    \end{cases2}
    and thus the topological derivative of the shape functional reads:
    \begin{equation}
        DJ(\Om)(E) =  \mathrm{sgn}_\Om(E) \mu_{E}(\left\{(g_2(u_\Om)-g_1(u_\Om))+(f_1-f_2)\right\}p_\Om).
    \end{equation}

\end{example}

We finish with the gradient-tracking example.

\begin{example}
    Consider $J(\Omega) = \F(u_\Omega)$ with the gradient-tracking function $\F(\cdot)$ of Example~\ref{ex:tracking_H1}:
    \begin{equation}
    \F(u)= \int_\Dsf |\nabla u -\nabla u_{\text{ref}}|^2 \;dx.
    \end{equation}
    The adjoint state is the (unique) solution $p_\Om \in W^{1,q^\prime}_0(\Omega)$, $\in [1,\frac{d}{d-1})$, of 
    \begin{cases2}
        -\Delta p_\Om+\partial_u\rho_\Om(x,u_\Om)p_\Om  & = -2\Delta(u_\Omega - u_{\text{ref}}) && \quad  \text{ in }\Dsf,\\
        p_\Om & = 0  && \quad \text{ on } \partial \Dsf,
\end{cases2}
and thus the derivative is again given by 
    \begin{equation}
        DJ(\Om)(E) =  \mathrm{sgn}_\Om(E) \mu_{E}(\left\{(g_2(u_\Om)-g_1(u_\Om))+(f_1-f_2)\right\}p_\Om).
    \end{equation}
\end{example}

\paragraph{Transmission problem}
We consider the transmission problem and cost function $J(\Omega) = G(u_\Omega)$ with $G$ defined in Example~\ref{ex:transmission}:
\begin{equation}\label{eq:cost_trans}
    G(u_\Omega):= \int_\Dsf (u_\Omega -u_{\text{ref}})^2 \;dx.
\end{equation}
We derived the following form of the topological derivative for $x_0\in \Dsf\setminus\partial \Omega$:
\begin{equation}\label{eq:def_DJ_trans}
    DJ(\Omega)(x_0,\omega) = 2\int_\Dsf U_0(u_\Omega - u_{\text{ref}})\;dx.
\end{equation}
Now introduce the adjoint associated with the cost functional \eqref{eq:cost_trans}, namely, $p_\Omega\in H^1_0(\Dsf)$ that solves in a weak sense $-\Div(\beta_\Omega\nabla p_\Omega) = -2(u_\Omega-u_{\text{ref}})$ in $\Dsf$. By \eqref{eq:limit_U_transmission} we know that $U_0$ solves
\begin{equation}\label{eq:very_weak_U0}
        \int_\Dsf U_0 v \;dx = \mathrm{sgn}_\Om(\{x_0\}) (\beta_2-\beta_1)\left(\frac{1}{|\omega|}\int_\omega \nabla K + \nabla u_\Omega(x_0)\; dx\right)\cdot \nabla \varphi_v(x_0)  \quad \text{ for all } v\in L^{q^\prime}(\Dsf),
    \end{equation}
    where $\varphi_v$ is the solution to \eqref{def:varphiv}. By definition of $\varphi_v$ we readily verify that 
    $p_\Omega = \varphi_v$ for $v:= -2(u_\Omega-u_{\text{ref}})$. Therefore testing \eqref{eq:very_weak_U0} with $v=-2(u_\Omega-u_{\text{ref}})$ yields together with \eqref{eq:def_DJ_trans}:
    \begin{equation}
        DJ(\Omega)(x_0,\omega) = -\mathrm{sgn}_\Om(\{x_0\}) (\beta_2-\beta_1)\left(\frac{1}{|\omega|}\int_\omega \nabla K + \nabla u_\Omega(x_0)\; dx\right)\cdot \nabla p_\Omega(x_0).
    \end{equation}
    In case $\omega=B_1(0)$, using Remark~\ref{rem:polarisation}, we can express the topological derivative of $J$  by 
    \begin{equation}
        DJ(\Omega)(x_0,\omega) = -\mathrm{sgn}_\Om(\{x_0\}) (\beta_2-\beta_1)(C_{\beta}+1) \nabla u_\Omega(x_0)\cdot \nabla p_\Omega(x_0). 
    \end{equation}

\section{Proofs for semilinear problems}
\subsection{Preliminary results on bilinear elliptic equations with measure data}
In this first section we give the basic regularity estimates for bilinear elliptic equations with measure data; the following proposition will be useful when considering the well-posedness of linearised systems. It should be noted that this result is not an immediate consequence of \cite{Ponce} or \cite{Stampacchia} but that the methods used to derive it is inspired by these contributions.
\begin{proposition}\label{Pr:Stampacchia}
Let $\Psi\in L^\infty(\Dsf)$ be such that the first eigenvalue $\lambda_1(\Psi)$ of the operator $L_\Psi:=-\Delta+\Psi$ is positive:
\[ \lambda_1(\Psi)=\inf_{u\in H^{1}_0(\Dsf)\backslash\{0\}} \frac{\int_\Dsf |\nabla u|^2\;dx+\int_\Dsf \Psi u^2\;dx}{\int_\Dsf u^2\;dx}>0.\] Then, for every $\mu\in \mathcal M(\Dsf)$ there exists a unique $u$ that satisfies
\begin{cases2}{Eq:Za}
    -\Delta u+\Psi u&=\mu&&\quad\text{ in }\Dsf, 
\\ u&=0&&\quad\text{ on }\partial \Dsf,
\end{cases2}
in the weak $W^{1,q}_0(\Dsf)$-sense for every $q\in [1,\frac{d}{d-1})$. Furthermore, for every $q\in [1,\frac{d}{d-1})$, there exists a constant $C_q$ such that, for every $\mu \in \mathcal M(\Dsf)$, 
\[ 
    \Vert u\Vert_{W^{1,q}(\Dsf)}\leq C_q \Vert \mu\Vert_{\mathcal M(\Dsf)}.
\]
\end{proposition}
\begin{proof}[Proof of Proposition \ref{Pr:Stampacchia}]
\underline{Approximation of \eqref{Eq:Za}:} 
We follow a standard \cite[Lemma 3.4]{Ponce} approximation scheme: for every $\mu \in \mathcal M(\Dsf)$ we consider a sequence $\{\mu_k\}_{k\in \VN}$ of $ C^\infty$ functions  that converges weakly (in the sense of measures) to $\mu$ and such that
\begin{equation}\label{Eq:Ze}\lim_{k\to \infty}\Vert \mu_k\Vert_{L^1(\Dsf)}=\Vert \mu\Vert_{\mathcal M(\Dsf)}.\end{equation} We consider the system \eqref{Eq:MainLinearisedSemiLinear} with $\mu$ replaced with $\mu_k$:
\begin{cases2}{Eq:PrK}
-\Delta v_k+\Psi  v_k&=\mu_k&&\quad\text{ in }\Dsf,
\\ v_k&=0&&\quad\text{ on } \partial\Dsf.
\end{cases2}
The existence of a solution $v_k$ to \eqref{Eq:PrK} follows from the minimisation of the energy functional 
\[ \mathscr E_k:H^{1}_0(\Dsf)\ni u\mapsto \frac12\int_\Dsf |\nabla u|^2\;dx+\frac12\int_\Dsf \Psi  u^2\;dx-\int_\Dsf \mu_k u\;dx.\]
To check that $\mathscr E_k$ is indeed coercive, we use the fact that $\lambda_1(\Psi)>0$ to obtain: for every $u\in H^{1}_0(\Dsf)$, 
\[\mathscr E_k(u)\geq \frac{\lambda_1(\Psi)}2\int_\Dsf u^2\;dx-\int_\Dsf \mu_ku\;dx.\] 

Regarding the uniqueness, observe that if there are two different solutions $(v_k,v_k')$  of \eqref{Eq:PrK} then the difference $z_k:=v_k-v_k'$ and satisfies
\begin{cases2}{Eq:Zk}
-\Delta z_k+\Psi  z_k&=0&&\quad\text{ in }\Dsf, 
\\ z_k&=0&&\quad\text{ on }\partial\Dsf.
\end{cases2}
Multiplying \eqref{Eq:Zk} by $z_k$ and integrating by parts we obtain 
\[ 
    \int_\Dsf |\nabla z_k|^2\;dx+\int_\Dsf \Psi  z_k^2\;dx=0.
\] 
Since $\lambda_1(\Psi)>0$, this is implies $z_k=0$, which means $v_k=v_k'$ and hence shows uniqueness.

\underline{Regularity estimates on the approximated problem:}
In order to derive $W^{1,p}$-estimates on the sequence $\{v_k\}_{k\in \VN}$, we begin with an \emph{a priori} $L^1$-estimate on $\{v_k\}_{k\in \VN}$; it will then suffice to apply the classical regularity result \cite[Proposition 4.1]{Ponce}. Consider, for every  function $h\in L^\infty(\Dsf)$, the solution $\theta_h$ of 
\begin{cases2}{Eq:Clap}
-\Delta \theta_h+\Psi\theta_h&=h&&\quad\text{ in }\Dsf, 
\\ \theta_h&=0&&\quad\text{ on }\partial\Dsf.
\end{cases2}
The existence and uniqueness of a solution to \eqref{Eq:Clap} follows from the same energy argument already used to obtain existence and uniqueness for $v_k$. We claim that for every $p\in [2,\infty)$,  there exists a constant $C_p$ such that 
\begin{equation}\label{Eq:Mila} 
    \Vert \theta_h\Vert_{W^{2,p}(\Dsf)}\leq C_p\Vert h\Vert_{L^p(\Dsf)}.
\end{equation}
\eqref{Eq:Mila} follows from a standard bootstrap argument which we just show the initialisation of. Using $\theta_h$ as a test function in the weak formulation of \eqref{Eq:Clap} we obtain 
\[
\int_\Dsf \theta_h^2\;dx\leq \frac1{\lambda_1(\Psi)}\int_\Dsf h^2\;dx,
\] 
whence elliptic regularity guarantees $\Vert \theta_h\Vert_{W^{2,2}(\Dsf)}\leq C \Vert h\Vert_{L^2(\Dsf)}$. This implies 
\eqref{Eq:Mila} for $p=2$. Now using the Sobolev embedding $W^{2,2}(\Dsf)\hookrightarrow L^p(\Dsf)$ ($p>2$) and writing  $-\Delta \theta_h = h-\Psi \theta_h\in L^p(\Dsf)$, $p>2$, we conclude again by elliptic regularity that $\theta_h\in W^{2,p}(\Dsf)$ and 
\[
    \|\theta_h\|_{W^{2,p}(\Dsf)}\le C(\|\theta_h\|_{L^p(\Dsf)}+ \|h\|_{L^p(\Dsf)}) \le C(\|\theta_h\|_{W^{2,2}(\Dsf)} + \|h\|_{L^p(\Dsf)}) \le C\|h\|_{L^p(\Dsf)},
\]
which is \eqref{Eq:Mila}.  Consequently, there exists a constant $C$, such that
\[ 
\Vert \theta_h\Vert_{L^\infty(\Dsf)}\leq C \Vert h\Vert_{L^\infty(\Dsf)}.
\]
Now, use $\theta_h$ as a test function in \eqref{Eq:PrK}. We obtain, integrating by parts twice,
\[ 
\left| \int_\Dsf v_k h\;dx\right|\leq \left|\int_\Dsf \theta_h \mu_k\;dx\right|\leq C \Vert \mu_k\Vert_{L^1(\Dsf)}\Vert h\Vert_{L^\infty(\Dsf)}.
\] 
By \eqref{Eq:Ze} we deduce that there exists a constant $C$ such that $ \sup_{k\in \VN}\Vert v_k\Vert_{L^1(\Dsf)}\leq C\Vert \mu\Vert_{\mathcal M(\Dsf)}. $ Observe now that we can rewrite the equation on $v_k$ as 
$ -\Delta v_k=\tilde \mu_k$ with 
$  \tilde \mu_k=\mu_k-\Psi v_k.$ As $\Psi\in L^\infty(\Dsf)$ we have, for a certain constant $C$,  that $\Vert \tilde\mu_k\Vert_{\mathcal M(\Dsf)}\leq C \Vert \mu\Vert_{\mathcal M(\Dsf)}$. Consequently, from \cite[Proposition~4.1]{Ponce} we know that, for every $q\in [1,\frac{d}{d-1})$, there exists a constant $c_q$ such that, for every $k\in \VN$, 
\[ 
\Vert v_k\Vert_{W^{1,q}(\Dsf)}\leq c_q \Vert \mu\Vert_{\mathcal M(\Dsf)}.
\]

We can thus extract a $W^{1,q}(\Dsf)$-weak, $L^q(\Dsf)$-strong, converging subsequence of $\{v_k\}_{k\in \VN}$. Let $U$ be the closure point under consideration. For every $p>d$, Sobolev embeddings imply that $W^{1,p}(\Dsf)\hookrightarrow C^0(\overline\Dsf)$ . Let $\theta\in W^{1,p}(\Dsf)$. Passing to the limit in the identity
\[ 
\int_\Dsf \nabla v_k\cdot\nabla \theta\;dx+\int_\Dsf \Psi  v_k \theta\;dx=\int_\Dsf \theta \mu_k\;dx
\]
it appears that $U$ solves \eqref{Eq:MainLinearisedSemiLinear} in the weak $W^{1,q}$-sense, and further satisfies the required regularity estimate. The existence of a solution is thus established.

\underline{Uniqueness of a solution to \eqref{Eq:MainLinearisedSemiLinear}:}
Assume that $u_1,u_1$ are two distinct solutions of \eqref{Eq:MainLinearisedSemiLinear}. Then $z:=u_1-u_2$ satisfies
\begin{cases2}{eq:z_unique}
\mathcal L_\Om z&=0&&\quad\text{ in }\Dsf,
\\ z&=0&&\quad\text{ on } \partial\Dsf.
\end{cases2}
It is then clear that $z$ solves \eqref{eq:z_unique} in the weak $H^{1}_0$-sense. However, we already proved (see above the proof of $z_k\equiv 0$) that this implies $z=0$ and thus the uniqueness. Thus the solution $u$ is necessarily unique.
\end{proof}

In this section we gather the proofs of Proposition \ref{Pr:LinearisedSemiLinear}, Theorem \ref{Th:MainSemiLinear}, Theorem \ref{Th:Main2SemiLinear}, Theorem \ref{thm:main_rates_rhs} and Corollary \ref{cor:symmetric_inclusion}.
\subsection{Proof of Proposition \ref{Pr:LinearisedSemiLinear}}
It suffices to prove that the potential 
\begin{equation}\label{Eq:WOm}W_\Om:=\frac{\partial \rho_\Om}{\partial u}(x,u_\Om)
\end{equation} 
is such that the assumptions of Proposition \ref{Pr:Stampacchia} are satisfied. Given that Assumption \eqref{Hyp:g} is satisfied, we know that 
\begin{equation}\label{Eq:WOmPo} 
W_\Om\geq 0.
\end{equation} 
Furthermore, by standard elliptic regularity, $u_\Om\in L^\infty(\Dsf)$, whence we conclude $W_\Om\in L^\infty(\Dsf).$ Consequently, the first eigenvalue $\lambda_1(W_\Om)$ (with the notations of Proposition \ref{Pr:Stampacchia}) is bounded from below:
\[ 
    \lambda_1(W_\Om)\geq \inf_{\substack{u\in H^{1}_0(\Dsf) \\ u \ne 0}} \frac{\int_\Dsf |\nabla u|^2\;dx}{\int_\Dsf u^2\;dx} > 0,
\] 
where the infimum on the right hand side is the first Dirichlet eigenvalue of the domain $\Dsf$. It suffices to use Proposition \ref{Pr:Stampacchia} to obtain the conclusion.

\subsection{Proof of Theorem \ref{Th:MainSemiLinear}}
The computations are similar whether we take $E\subset \Om$ or $E\subset \Dsf\backslash\overline \Om$. Thus, for notational simplicity, we consider the case $E\subset \Dsf\backslash \overline\Om$.
The function $U_\eps$ solves: $U_\eps =0 $ on $\partial \Dsf$ and in a weak $W^{1,q}_0(\Dsf)$ sense:
\[
    -\Delta U_\eps+\chi_\Om\frac{(g_1(u_\eps)-g_1(u_0)-g_2(u_\eps)+g_2(u_0))}{|E_\eps|}+\mu_\eps(g_1(u_\eps)-g_2(u_\eps))+\frac{g_2(u_\eps)-g_2(u_0)}{|E_\eps|}=(f_1-f_2)\mu_\eps.
\]
Now observe that, as $\chi_{\Om_\eps(E)}\underset{\eps\searrow 0}\to \chi_\Om$ in $L^2(\Dsf)$, elliptic regularity estimates entail 
\begin{equation}\label{Eq:Cv}
u_\eps\underset{\eps\searrow0}\to u_0\text{ in } C^0(\overline\Dsf).
\end{equation} 
By the mean value theorem, for $i=1,2$, there exists $v_{\eps,i}\in (u_0;u_\eps)$ or $(u_\eps;u_0)$ such that
\[ 
g_i(u_\eps)-g_i(u_0)=-g_i'(v_{\eps,i})(u_\eps-u_0)\quad (i=1,2).
\] From \eqref{Eq:Cv} we have $v_{\eps,i}\underset{\eps\searrow 0}\to u_0$ in $C^0(\overline\Dsf)$. This allows to rewrite the equation on $U_\eps$ as 
\begin{equation}\label{eq:U_eps_proof}
    -\Delta U_\eps+\chi_\Om{g_1'(v_{\eps,1})U_\eps+\chi_{\Dsf\backslash\Om}g_2'(v_{\eps,2})}U_\eps=-\mu_\eps(g_1(u_\eps)-g_2(u_\eps))+(f_1-f_2)\mu_\eps \quad \text{ in } \Dsf.
\end{equation}
From the same regularity estimates derived in the proof of Proposition \ref{Pr:LinearisedSemiLinear} we deduce that, for every $q\in [1,\frac{d}{d-1})$ and $\delta >0$ small,
\[ 
    \sup_{\eps\in (0,\delta] }\Vert U_\eps\Vert_{W^{1,q}(\Dsf)}<\infty.
\] 
We may thus pass to the weak $W^{1,q}$, strong $L^q$ limit in the equation of $U_\eps$ to obtain that every  accumulation  point of this sequence is a weak $W^{1,q}$-solution to \eqref{Eq:UeSemiLinear}. Since the uniqueness of a solution to this equation was established in Proposition \ref{Pr:LinearisedSemiLinear} the conclusion follows, if we can prove that the convergence is, in fact, strong in $W^{1,q}(\Dsf)$. Here, we use  \cite[Assertion (21)]{Boccardo} (see also \cite[Proposition 4.9]{Ponce}). Rewrite the equation \eqref{eq:U_eps_proof} on $U_\eps$ as 
\[ 
-\Delta U_\eps=(f_1-f_2-g_1(u_\eps)+g_2(u_\eps))\mu_\eps-(\chi_\Om{g_1'(v_{\eps,1})+\chi_{\Dsf\backslash\Om}g_2'(v_{\eps,2})})U_\eps=\tilde \mu_\eps.
\] 
From the strong convergence of $U_\eps$ in $L^1(\Dsf)$, we deduce from  \cite[Assertion (21)]{Boccardo} that $\{U_\eps\}_{\eps\in (0,\delta]}$ is a sequentially compact family in $W^{1,q}_0(\Dsf)$ and thus converges strongly. The proof of the proposition is complete.

\subsection{Proof of Theorem \ref{Th:Main2SemiLinear}} It will be convenient to observe that when $E=\{x_0\}$ the function $U_0$ solves  the equation  
\begin{cases2}{Eq:Ux0}
-\Delta U_0+W_\Om U_0&=\left(\mathrm{sgn}_\Om(\{x_0\}) F(x_0)\right)\delta_{x_0}&&\quad\text{ in }\Dsf, 
\\ U_0&=0&&\quad\text{ on }\partial \Dsf,
\end{cases2} 
in a weak $W^{1,q}_0(\Dsf)$-sense. The function $F$ in \eqref{Eq:Ux0} is defined as 
\[
F=g_2(u_\Om)-g_1(u_\Om)+f_1-f_2\in L^\infty(\Dsf),
\] 
and the potential $W_\Om$ is defined in \eqref{Eq:WOm}. Introduce the Green kernel $G=G_\Om=G_\Om(x,y)$ of the operator $-\Delta +W_\Om$, that is, the unique solution of 
\begin{cases2}{Eq:GreenKernel}
    -\Delta_x G(y,x)+W_\Om G &=\delta_{x=y}&&\quad\text{ in }\Dsf, 
\\G(y,x)&=0&&\quad\text{ on }\partial \Dsf.
\end{cases2}  
A detailed study of the Green kernel of operators $L$ having the form $-\sum_i\partial_i(\sum_j a_{i,j}\partial_j)$ can be found in the seminal work \cite{Littman}. Here, the existence of a Green kernel follows from Proposition \ref{Pr:Stampacchia}, and the Green kernel is symmetric in the sense that $G(x,y)=G(y,x)$ for all for all $x\neq y\in \Dsf$. Furthermore, from Proposition~\ref{Pr:Stampacchia} and Sobolev embeddings we know that, 
\begin{equation}\label{Eq:GKL2}
\forall y\in \Dsf\,, G(y,\cdot)\in L^2(\Dsf).
\end{equation}
Finally, it is clear that for all $y\in \Dsf\,, U_{\{y\},0}=\mathrm{sgn}_\Om(\{y\})F(y) G(y,\cdot)$ and consequently, 
for every $h\in L^2(\Dsf)$, the function 
\[
\tilde u= x\mapsto \int_\Dsf \mathrm{sgn}_\Om(\{y\}) U_{\{y\},0}(x)h(y)dy=\int_\Dsf F(y) G(x,y)h(y)dy,
\] 
is well-defined. It suffices to differentiate it to obtain that $\tilde u$ solves \eqref{Eq:Udot}, thereby concluding the proof of Theorem \ref{Th:Main2SemiLinear}.

\subsection{Proof of Theorem~\ref{thm:main_rates_rhs}}

\begin{proof}
    To establish (i) for $d=2$ we have $K(x) = R(x) + O(|x|^{-1})$ as $|x|\to\infty$ and thus for $x\ne x_0$:
    \begin{equation}
        K(T_\eps^{-1}(x)) = R(T_\eps^{-1}(x)) + O(|T_\eps^{-1}(x)|^{-1}).
    \end{equation}
    Since $R(x) = b\ln(|x|)$, it follows $R(T_\eps^{-1}(x)) = b\ln(x-x_0) - b\ln(\eps) = R(x-x_0) - b\ln(\eps)$ and thus we conclude for $\eps \searrow 0$:
    \begin{equation}
        \frac{1}{|\omega_\eps|} \eps^2(K(T_\eps^{-1}(x)) + v(x) + \ln(\eps)b) =  |\omega|^{-1}( b\ln(|x-x_0|) + v(x)) + O(|T_\eps^{-1}(x)|^{-1}),
    \end{equation}
    in view of $|T_\eps^{-1}(x)|^{-1} = \eps |x-x_0|^{-1}$ the result follows. The proof of \eqref{eq:pointwise_limit_rhs_dim3} is established the same way and left to the reader.

For the sake of simplicity we only give a proof for item (iii), that is we restrict ourselves to dimension $d\ge3$. Note that the same arguments can be used to show the according results of item (ii).
We start by proving the last item of (iii). Therefore, first note that we have
\[U_\eps - |\omega|^{-1}( \eps^{2-d}K\circ T_\eps^{-1} + v)=|\omega|^{-1}\eps^{2-d}\left(K_\eps\circ T_\eps^{-1}-K\circ T_\eps^{-1}-\eps^{d-2}v\right).\]
Thus, from \eqref{eq:main_3} and changing variables, we obtain
\begin{equation}\label{eq:H1_est_interior}
    \|U_\eps - |\omega|^{-1}( \eps^{2-d}K\circ T_\eps^{-1} + v)\|_{H^1(\Dsf)} \le C\eps.
\end{equation}
Now we estimate for $p\in (1,\frac{d}{d-1})$ by the triangle inequality and H\"older's inequality:
\begin{align}
    \begin{split}
        \|\nabla (U_\eps - |\omega|^{-1}(R(x-x_0)+v))\|_{L^p(\Dsf)^d}  \le &  C\|\nabla (U_\eps - |\omega|^{-1}(\eps^{2-d}K\circ T_\eps^{-1}+v))\|_{L^2(\Dsf)^d} \\
                                                                         & + C|\omega|^{-1}\|\nabla (\eps^{2-d}K\circ T_\eps^{-1} -R(x-x_0))\|_{L^p(\Dsf)^d}.
\end{split}
\end{align}
Using the estimate \eqref{eq:H1_est_interior}, we see that the first term on the right hand side is bounded by $C\eps$.  For the second term we note
 $\nabla R(\eps x) = \eps^{1-d}\nabla R(x)$ and thus
\begin{align}
    \int_\Dsf |\nabla (\eps^{2-d}K\circ T_\eps^{-1} - R(x-x_0))|^p \;dx & = \int_\Dsf |\eps^{1-d} \nabla K(T_\eps^{-1}(x)) - \nabla R(x-x_0))^p \;dx \\
                                                       & = \int_{\Dsf_\eps}\eps^d |\eps^{1-d}K(x) - \nabla R(\eps x)|^p \;dx \\
                                                       & = \int_{\Dsf_\eps}\eps^{d-p(d-1)} |K(x) - \nabla R(x)|^p \;dx,
\end{align}
and 
\begin{equation}
    \|\nabla (K\circ T_\eps^{-1} -R(x-x_0))\|_{L^p(\Dsf)^d}^p = \eps^{d-p(d-1)} \|\nabla(K-R)\|_{L^p(\Dsf_\eps)^d}^p.
\end{equation}
Now according to Lemma \ref{lem:decay_KR_integrability}, we have 
\begin{equation}
    \|\nabla(K-R)\|_{L^p(\Dsf_\eps)^d} \le \|\nabla(K-R)\|_{L^p(\VR^d)^d}<\infty \quad \text{ for } p \in \left(1,\frac{d}{d-1}\right). 
\end{equation}
Therefore we obtain for $p\in (1,\frac{d}{d-1})$ noticing that $0<p^{-1}(d-p(d-1))<1$:
\begin{equation}
    \|\nabla (U_\eps - |\omega|^{-1}(R(x-x_0)+v))\|_{L^p(\Dsf)^d} \le C\eps + C\eps^{\frac{d-p(d-1)}{p}} \le C\eps^{\frac{d-p(d-1)}{p}}.
\end{equation}
Since by \eqref{eq:vk}, $v(x)=-R(x-x_0)$ and $u_\eps(x)=u_0(x)=0$ for $x\in \partial \Dsf$, we have 
\[
    U_\eps - |\omega|^{-1}(R(x-x_0)+v) = |\omega|^{-1}(R(x-x_0)-R(x-x_0)) =0 \quad \text{ on } \partial \Dsf.
\]
Therefore the Poincar\'e inequality yields
\[
C\|U_\eps - |\omega|^{-1}(R(x-x_0)+v)\|_{L^p(\Dsf)} \le  \|\nabla (U_\eps - |\omega|^{-1}(R(x-x_0)+v))\|_{L^p(\Dsf)^d},
\]
and hence the last item of (iii) follows.

We now prove the first item in (iii).  We compute for $d\ge 3$ and $p\in (\frac{d}{d-1},\frac{d}{d-2})$, using the  continuous embedding $H^1(\Dsf) \stackrel{c}{\hookrightarrow} L^p(\Dsf)$  for $p\in [1,\frac{d}{d-2})$:
\begin{align}\label{eq:rhs_estimate_UepsLp3}
    \begin{split}
        \|U_\eps - |\omega|^{-1}(R(x-x_0)+v)\|_{L^p(\Dsf)}  \le &  C\|U_\eps - |\omega|^{-1}(\eps^{2-d}K\circ T_\eps^{-1}+v)\|_{L^p(\Dsf)} \\
                                                                & + C|\omega|^{-1}\|\eps^{2-d}K\circ T_\eps^{-1}-R(x-x_0)\|_{L^p(\Dsf)}\\
    \le &  C\|U_\eps - |\omega|^{-1}(K\circ T_\eps^{-1}+v)\|_{H^1(\Dsf)} \\
        & + C|\omega|^{-1}\|\eps^{2-d}K\circ T_\eps^{-1}-R(x-x_0)\|_{L^p(\Dsf)}.
\end{split}
\end{align}
Moreover, we have by changing variables:
\begin{equation}
    \|\eps^{2-d}K\circ T_\eps^{-1}-R(x-x_0)\|_{L^p(\Dsf)} = \eps^{\frac{d-p(d-2)}{p}}\|K-R\|_{L^p(\Dsf_\eps)},
\end{equation}
and according to Lemma \ref{lem:decay_KR_integrability}
\begin{equation}
    \|K-R\|_{L^p(\Dsf_\eps)}\le \|K-R\|_{L^p(\VR^d)}<\infty \quad \text{ for }\quad p\in \left(\frac{d}{d-1},\frac{d}{d-2}\right).
\end{equation}
Therefore from \eqref{eq:rhs_estimate_UepsLp3} and \eqref{eq:H1_est_interior}, we have for $p\in \left(\frac{d}{d-1},\frac{d}{d-2}\right)$
\begin{equation}
  \|U_\eps - |\omega|^{-1}(R(x-x_0)+v)\|_{L^p(\Dsf)}  \le C\eps + C\eps^{\frac{d-p(d-2)}{p}}\le C\eps^{\frac{d-p(d-2)}{p}}.
\end{equation}

\end{proof}

\subsection{Proof of Corollary~\ref{cor:symmetric_inclusion}}

\begin{proof}
    This is a direct consequence of Theorem~\ref{thm:main_rates_rhs} having a close inspection of the proof of Theorem~\ref{thm:main_rates_rhs}, item(ii). Indeed, since $\omega=B_1(0)$, we have according to \cite{a_BAGAST_2021b} a.e. in $\Dsf$:
\begin{equation}
    U_\eps - |\omega|^{-1}(K\circ T_\eps^{-1}+v +b\ln(\eps))=0 \quad \text{ for } d=2,
\end{equation}
and 
\begin{equation}\label{eq:exact_3D}
    U_\eps - |\omega|^{-1}(\eps^{2-d}K\circ T_\eps^{-1}+v)=0 \quad \text{ for } d\ge 3.
\end{equation}
Moreover, $K=R$ on $\VR^d\setminus\overline{B_1(0)}$, that is, the asymptotics of $K$ aborts after the first term. Therefore for $d\ge 3$ it follows from \eqref{eq:exact_3D}:
\begin{align}
    \|U_\eps - |\omega|^{-1}(R(x-x_0)+v)\|_{L^p(\Dsf)} \le & C|\omega|^{-1}\|\eps^{2-d}K\circ T_\eps^{-1}-R(x-x_0)\|_{L^p(B_\eps(0))}.
\end{align}
Now changing variables $T_\eps(x)=y$ and $R(T_\eps(x)-x_0)=\eps^{-1}R(x)$ shows that 
\begin{equation}
\|\eps^{2-d}K\circ T_\eps^{-1}-R(x-x_0)\|_{L^p(B_\eps(0))} = \eps^{\frac{d-p(d-2)}{p}}\|K-R\|_{L_p(B_1(0))} 
\end{equation}
the last term is finite for $p\in [1,\frac{d}{d-1})$ since $R(x)=E(x)=|\omega|c_d|x|^{-(d-2)}\in L_p(B_1(0))$ for such $p$ (recall $E$ was defined in \eqref{eq:fundamental_solution}). 
The case $d=2$ is treated in the same fashion noting that for $d=2$ we have $R(x)=-|\omega|(2\pi)^{-1}\ln(|x|)\in L_p(B_1(0))$ for all $p\ge 1$.
\end{proof}

\section{Proofs for the transmission problem}
\subsection{Proof of Proposition \ref{Pr:LinearisedTransmission}}

\underline{Existence of a very weak solution}
To establish the existence of a solution, it suffices to observe that the linear map 
\[
T:L^p(\Dsf)\ni v\mapsto \zeta \cdot \nabla \varphi_v(x_0) 
\] is continuous for every $p>d$, as 
 $\Vert \varphi_v\Vert_{ C^1(\overline\Omega\cup(\Dsf\setminus\overline\Omega))}\leq C \Vert v\Vert_{L^p(\Dsf)}$ by elliptic regularity.  By the Riesz representation theorem, there exists a unique $\varphi_{\zeta,x_0}\in L^{p'}(\Dsf)$, such that
\begin{equation}\label{eq:very_weak_equation}
    \int_\Dsf \varphi_{\zeta,x_0}v  = T(v)  \quad \text{ for every } v\in L^p(\Dsf).
\end{equation}
Therefore \eqref{eq:weak2} admits a unique solution $u\in L^{p'}(\Dsf)$, which further satisfies the required regularity estimates.

\subsection{Proof of Theorem~\ref{Th:MainTransmission}}
\begin{proof}[Proof of Theorem~\ref{Th:MainTransmission}]
Let $u_\eps:= u_{\Omega_\eps(x_0,\omega)}$, $u_0:=u_\Om$ and  $U_\eps = \frac{u_\eps-u_0}{|\omega_\eps|}$. 
Then we obtain 
\begin{equation}
    \int_\Dsf \beta_\Om \nabla (u_\eps - u_0)\cdot \nabla \varphi \;dx = \mathrm{sgn}_\Om(\{x_0\}) (\beta_2-\beta_1)\left(\int_{\omega_\eps} \nabla (u_\eps-u_0) \cdot \nabla \varphi \;dx + \int_{\omega_\eps} \nabla u_0(x)\cdot \nabla \varphi \;dx \right),
\end{equation}
 for all $\varphi\in H^1_0(\Dsf)$. 
Dividing by $|\omega_\eps| = |\omega|\eps^d$ and using $K_\eps = \frac{(u_\eps-u_0)\circ T_\eps}{\eps}$, this can be written as 
\begin{align}
    \int_\Dsf \beta_\Om \nabla U_\eps \cdot \nabla \varphi \;dx =& \mathrm{sgn}_\Om(\{x_0\}) (\beta_2-\beta_1)\frac{1}{|\omega_\eps|}\int_{\omega_\eps} \nabla K_\eps\circ T_\eps^{-1} \cdot \nabla \varphi(x) \;dx\\ &+ \mathrm{sgn}_\Om(\{x_0\}) (\beta_2-\beta_1)\frac{1}{|\omega_\eps|}\int_{\omega_\eps} \nabla u_0(x)\cdot \nabla \varphi(x) \;dx,
\end{align}
for all $\varphi\in H^1_0(\Dsf)$. Now choosing $\varphi = \varphi_v$ (with $\varphi_v$ defined in \eqref{def:varphiv} for $v\in L^p(\Dsf)$, $p>d$ and $q:=p'=q/(q-1)$) as a test function and integrating by parts in the first integral using $-\Div(\beta_\Omega \nabla \varphi_v) = v$ yields the very weak formulation:
\begin{align}
    \int_\Dsf  U_\eps v \;dx =&  \mathrm{sgn}_\Om(\{x_0\}) (\beta_2-\beta_1)\frac{1}{|\omega_\eps|}\int_{\omega_\eps} \nabla K_\eps\circ T_\eps^{-1} \cdot \nabla \varphi_v(x) \;dx\\ &+ \mathrm{sgn}_\Om(\{x_0\}) (\beta_2-\beta_1)\frac{1}{|\omega_\eps|}\int_{\omega_\eps} \nabla u_0(x)\cdot \nabla \varphi_v(x) \;dx,
\end{align}
for all $v\in L^p(\Dsf)$.
Subtracting the limit equation for $U_0$ yields
\begin{align}
    \int_\Dsf  (U_\eps-U_0) v \;dx =&  \mathrm{sgn}_\Om(\{x_0\}) (\beta_2-\beta_1)\frac{1}{|\omega_\eps|}\int_{\omega_\eps} (\nabla K_\eps\circ T_\eps^{-1} - \nabla K\circ T_\eps^{-1})\cdot \nabla \varphi_v(x) \;dx\\ 
                                    &+ \mathrm{sgn}_\Om(\{x_0\}) (\beta_2-\beta_1)\frac{1}{|\omega_\eps|}\int_{\omega_\eps} \nabla K\circ T_\eps^{-1} \cdot (\nabla \varphi_v(x)-\nabla \varphi_v(x_0)) \;dx\\
                                    &+ \mathrm{sgn}_\Om(\{x_0\}) (\beta_2-\beta_1)\frac{1}{|\omega_\eps|}\int_{\omega_\eps} (\nabla u_0(x) - \nabla u_0(x_0))\cdot \nabla \varphi_v(x) \;dx \\
                                    &+ \mathrm{sgn}_\Om(\{x_0\}) (\beta_2-\beta_1)\frac{1}{|\omega_\eps|}\int_{\omega_\eps} \nabla u_0(x)\cdot (\nabla \varphi_v(x)-\nabla \varphi_v(x_0)) \;dx \\
                                    & =: I_1(v)+I_2(v)+I_3(v)+I_4(v).
\end{align}
It is readily checked that using $\|\varphi_v\|_{C^1(\Omega\cup\Dsf\setminus\overline\Omega)}\le \|v\|_{L_p(\Dsf)}$:
\begin{equation}
    |I_1(v)|\le C\|\nabla K_\eps-\nabla K\|_{L_1(\omega)}\|\nabla \varphi_v\|_{C^0(\overline{\omega_\eps})^d} \le C\|\nabla K_\eps-\nabla K\|_{L_1(\omega)^d}\|v\|_{L^p(\Dsf)} 
\end{equation}
and
\begin{equation}
    |I_3(v)| \le C|\omega_\eps|^{-1}\|\nabla u_0-\nabla u_0(x_0)\|_{L^1(\omega_\eps)^d}\|v\|_{L^p(\Dsf)}. 
\end{equation}
Now we recall the following equation which follows from the proof of Morrey's inequality \cite[p.280, Theorem~4]{b_EV_2010a}: for all $\varphi \in W^{2,p}(\Omega\cup \Dsf\setminus\overline\Omega)$, $p>d$: 
\begin{equation}
    \frac{1}{|B_\eps(x_0)|}\int_{B_\eps(x_0)}|\nabla \varphi(x)-\nabla \varphi(x_0)|\;dx \le C\int_{B_\eps(x_0)}\frac{|\nabla^2\varphi(x)|}{|x-x_0|^{d-1}}\;dx.
\end{equation}
Now there are constants $\varrho>1$ and $C>0$, such that $|B_{\eps\varrho}(x_0)|\le C|\omega_{\eps}(x_0)|$  and $\omega_\eps(x_0)\subset B_{\eps\varrho}(x_0)$ for $\eps >0$. Therefore,
\begin{equation}
    \frac{1}{|\omega_\eps(x_0)|}\int_{\omega_\eps(x_0)}|\nabla \varphi(x)-\nabla \varphi(x_0)|\;dx \le C \frac{1}{|B_{\eps\varrho}(x_0)|}\int_{B_{\eps\varrho}(x_0)}|\nabla \varphi(x)-\nabla \varphi(x_0)|\;dx.
\end{equation}
Therefore it follows from H\"older's inequality:
\begin{equation}
    |I_2(v)|\le C\|\nabla K\|_{C^0(\overline\omega)^d}\int_{B_{\eps\varrho}(x_0)}\frac{|\nabla^2\varphi_v(x)|}{|x-x_0|^{d-1}}\;dx\le C\|\nabla K\|_{C^0(\overline\omega)^d}\|v\|_{L^p(\Dsf)} \left(\int_{B_{\eps\varrho}(x_0)}\frac{1}{|x-x_0|^{p'(d-1)}}\;dx\right)^{1/p'}.
\end{equation}
Changing variables yields
\begin{equation}
    \left(\int_{B_{\eps\varrho}(x_0)}\frac{1}{|x-x_0|^{p'(d-1)}}\;dx\right)^{1/p'} = (\varrho\eps)^{\frac{d-p'(d-1)}{p'}}\left(\int_{B_1(0)}\frac{1}{|x|^{p'(d-1)}}\;dx\right)^{1/p'}
\end{equation}
and the last integral is finite if $p'=q\in (1,\frac{d}{d-1})$, which is satisfied since $p>d$.
Similarly we can show that 
\begin{equation}
    |I_4(v)|\le C\eps^{\frac{d-p'(d-1)}{p'}}\|v\|_{L^p(\Dsf)}.
\end{equation}
Summarising we have shown that (recall $p'=q$)
\begin{equation}
    \|U_\eps-U_0\|_{L^q(\Dsf)}\le C(\eps^{\frac{d-q(d-1)}{q}} + \|\nabla K_\eps-\nabla K\|_{L_1(\omega)^d} + |\omega_\eps|^{-1}\|\nabla u_0-\nabla u_0(x_0)\|_{L^1(\omega_\eps)^d})
\end{equation}
and the in view of right differentiability of $\nabla u_0$ near $x_0$, $|\omega_\eps|^{-1}\|\nabla u_0-\nabla u_0(x_0)\|_{L^1(\omega_\eps)^d}\le C\eps$ and the estimate \eqref{eq:trans_estimate}. Moreover, the right hand side goes to zero as $\eps\searrow 0$ in view of Lemma~\ref{lem:convergence_Keps_transmission}.
\end{proof}%

\subsection{Proof of Theorem~\ref{thm:mainThmTransmissionAsymptotic}}
\begin{proof}[Proof of Theorem~\ref{thm:mainThmTransmissionAsymptotic}]
   We first note that multiplying \eqref{eq:asymp_exp_trans_fixed} with $\eps$ gives
\begin{equation}
    \|u_\eps - u_0 -  \eps K\circ T_\eps^{-1} - \eps^dv\|_{W^{1,q}(\Dsf)} \le  C\eps^{1+\frac{d}{q}},
\end{equation} 
and this yields by division by $|\omega_\eps|$ with the definition $u_\eps:= u_{\Omega_\eps(x_0,\omega)}$  and  $U_\eps = \frac{u_\eps-u_0}{|\omega_\eps|}$ the estimate
\begin{equation}\label{eq:estimate_Ueps_D}
    \|U_\eps -  \eps |\omega_\eps|^{-1}K\circ T_\eps^{-1} - |\omega|^{-1}v\|_{W^{1,q}(\Dsf)} \le  C\eps^{1+\frac{d}{q}-d},
\end{equation}
where we note that the exponent $1+\frac{d}{q}-d>0$ for $q\in (1,\frac{d}{d-1})$.  Now we estimate 
\begin{equation}\label{eq:esitmate_Ueps_trans_layer}
    \begin{split}
        \|U_\eps - |\omega|^{-1}(R + v) \|_{L^q(\Dsf)}   \le &  \|U_\eps - \eps |\omega_\eps|^{-1}K\circ T_\eps^{-1} - |\omega|^{-1}v\|_{L^q(\Dsf)} \\
                                                                   & +  \|\eps |\omega_\eps|^{-1}K\circ T_\eps^{-1} - |\omega|^{-1}R\|_{L^q(\Dsf)}.
\end{split}
\end{equation}
The first term on the right hand side is bounded in view of \eqref{eq:estimate_Ueps_D}. To treat the second term on the right hand side of \eqref{eq:esitmate_Ueps_trans_layer}, we 
change variables, recall $\Dsf_\eps = T_\eps^{-1}(\Dsf)$, and $R\circ T_\eps = \eps^{-(d-1)} R$ to obtain
\begin{equation}
    \|\eps |\omega_\eps|^{-1}K\circ T_\eps^{-1} - |\omega|^{-1}R\|_{L^q(\Dsf)}^q = \eps^{d-q(d-1)} |\omega|^{-q}\|K- R\|_{L^q(\Dsf_\eps)}^q.
\end{equation}
Now we recall $R(x) = \frac{Ax}{|x|^d}, K \in L^q(\omega)$ for $q\in [1,\frac{d}{d-1})$. 
In addition, since 
$K(x)-R(x)$ behaves as $|x|^{-d}$ for $|x|\to\infty$, we also have (similarly to Lemma~\ref{lem:decay_KR_integrability}) that $\|K-R\|_{L^q(\VR^d\setminus \omega)}$ is bounded. It follows that 
\begin{equation}
    \|\eps |\omega_\eps|^{-1}K\circ T_\eps^{-1} - |\omega|^{-1}R\|_{L^q(\Dsf)}^q \le  \eps^{d-q(d-1)} |\omega|^{-q}\|K- R\|_{L^q(\VR^d)}^q.
\end{equation}
Finally noting that $\frac{d-q(d-1)}{q}<1$ is equivalent to $q>1$ this finishes the proof.

\end{proof}

\section*{Appendix}

\paragraph{Proof of Lemma \ref{lma:expansion_rhs}}
\begin{proof}
We start with $d=2$ and aim to derive an equation for $V_\eps:=K_\eps - K - v\circ T_\eps-\ln(\eps)b$.
Subtracting the weak formulation from the perturbed state equation \eqref{eq:perturbed_state_rhs} for $\eps>0$ and $\eps=0$ yields
\begin{equation*}
\int_\Dsf\nabla (u_\eps-u_0)\cdot\nabla\varphi\;dx=\int_{\omega_\eps}(f_1-f_2)\varphi\;dx,\quad\text{ for all }\varphi\in H^1_0(\Dsf).
\end{equation*}
Now the change of variables $T_\eps(x)=y$, $\Dsf_\eps=T_\eps^{-1}(\Dsf)$,   and dividing by $\eps^2$ shows
\begin{equation*}
\int_{\Dsf_\eps}\nabla K_\eps\cdot\nabla\varphi\;dx=\int_{\omega}(f_1-f_2)\varphi\;dx,\quad\text{ for all }\varphi\in H^1_0(\Dsf_\eps).
\end{equation*}
Note that, by extending $\varphi\in H^1_0(\Dsf_\eps)$ by 0 in the exterior, we get
\begin{equation*}
\int_{\Dsf_\eps}\nabla K\cdot\nabla\varphi\;dx=\int_{\omega}(f_1-f_2)\varphi\;dx,\quad\text{ for all }\varphi\in H^1_0(\Dsf_\eps).
\end{equation*}
Hence, we conclude
\begin{align*}
\int_{\Dsf_\eps}\nabla V_\eps\cdot\nabla\varphi\;dx=0 ,\quad\text{ for all }\varphi\in H^1_0(\Dsf_\eps),
\end{align*}
where we used that $v$ is harmonic. Furthermore, we see that for $x\in\partial\Dsf_\eps$ we have
\begin{equation*}
V_\eps(x)=-K(x)-v\circ T_\eps(x)-\ln(\eps) b=-K(x)+R(\eps x)-\ln(\eps)b=R(x)-K(x),
\end{equation*}
where we used \eqref{eq:asymptotic_K_rhs} and $b\in\VR$ is chosen such that $R(\eps)=b\ln(\eps)$. Thus we observe, for $\eps>0$ sufficiently small, there holds
\begin{equation*}
|V_\eps(x)|\le c|x|^{-1}+O(|x|^{-2}),\quad\text{ for all }x\in\partial\Dsf_\eps.
\end{equation*}
Finally, we can apply \cite[Lemma 3.4, Lemma 3.7]{a_BAST_2021a} to conclude
\begin{equation*}
\|V_\eps\|_\eps\le C \left(\eps^{\frac{1}{2}}\|V_\eps\|_{L^2(\partial \Dsf_\eps)}+\|V_\eps\|_{H^{\frac{1}{2}}(\partial \Dsf_\eps)}\right)\le C\eps.
\end{equation*}
The proof for dimension $d\ge3$ is similar. An identical computation shows that $V_\eps:=K_\eps - K - \eps^{d-2}v\circ T_\eps$ satisfies
\begin{align*}
\int_{\Dsf_\eps}\nabla V_\eps\cdot\nabla\varphi\;dx=0 ,\quad\text{ for all }\varphi\in H^1_0(\Dsf_\eps),
\end{align*}
and  one readily checks that for a.e. $x\in \partial\Dsf_\eps$ there holds
\begin{equation*}
V_\eps(x)=-K(x)-\eps^{d-2}v\circ T_\eps(x)=-K(x)+\eps^{d-2}R(\eps x)=R(x)-K(x),
\end{equation*}
where in the last equality we used that $R$ is homogenous of degree $-(d-2)$; see \eqref{eq:R_rhs}. Thus, an application of \cite[Lemma 3.4, Lemma 3.7]{a_BAST_2021a} yield
\begin{equation*}
\|V_\eps\|_\eps\le C \left(\eps^{\frac{1}{2}}\|V_\eps\|_{L^2(\partial \Dsf_\eps)}+\|V_\eps\|_{H^{\frac{1}{2}}(\partial \Dsf_\eps)}\right)\le C\eps^{\frac{d}{2}}.
\end{equation*}
\end{proof}

\paragraph{Proof of Lemma \ref{lem:convergence_Keps_transmission}}
\begin{proof}
First note that, testing with $\varphi\in H^1_0(\Dsf)$, we can rewrite \eqref{eq:equation_K} as
\begin{equation}
\int_{\Dsf_\eps}\beta_{\omega\cup T_\eps^{-1}(\Omega)}\nabla K\cdot\nabla\varphi\;dx=\mathrm{sgn}_\Om(\{x_0\})(\beta_2-\beta_1)\int_\omega\nabla u_0(x_0)\cdot\nabla\varphi\;dx+\mathrm{sgn}_\Om(\{x_0\})(\beta_2-\beta_1)\int_{T_\eps^{-1}(\partial\Omega)}\partial_\nu K\varphi\;dS.
\end{equation}
Hence, combined with the rescaled equations \eqref{strongtrans} and \eqref{eq:corr_v_transmission}, we see that $V_\eps:=K_\eps-K-\eps^{d-1}v\circ T_\eps$ satisfies
\begin{align}
\int_{\Dsf_\eps}\beta_{\omega\cup T_\eps^{-1}(\Omega)}\nabla V_\eps\cdot\nabla\varphi\;dx=&\mathrm{sgn}_\Om(\{x_0\})(\beta_2-\beta_1)\int_\omega\left(\nabla u_0\circ T_\eps-\nabla u_0(x_0)\right)\cdot\nabla\varphi\;dx\\
&+\mathrm{sgn}_\Om(\{x_0\})(\beta_2-\beta_1)\int_{T_\eps^{-1}(\partial\Omega)}\partial_\nu (K-R)\varphi\;dS=:F_\eps(\varphi).
\end{align}
and $V_\eps=K-R=:g_\eps$ on $\partial \Dsf_\eps$. From \cite[Lemma 3.7]{a_BAST_2021a} (in an $L^p$ setting) we deduce
\begin{equation}
\|V_\eps\|_{L^q(\Dsf_\eps)}+\|\nabla V_\eps\|_{L^q(\Dsf_\eps)^d}\le C\left( \|F_\eps\|_{L^q(\Dsf_\eps)}+\eps^{1-\frac{1}{q}}\|g_\eps\|_{L^q(\partial\Dsf_\eps)}+|g_\eps|_{W^{1-\frac{1}{q},q}(\partial\Dsf_\eps)}\right).
\end{equation}
In view of $|K(x)-R(x)|\approx|x|^{-d}$ for $|x|\to\infty$ and the scaling properties of Lemma \ref{lem:scaling_bnd_transmission} and Lemma \ref{lem:scaling_norms_q}, the result follows.
\end{proof}

\begin{lemma}\label{lem:decay_KR_integrability}
Let $\omega\subset\VR^d$ such that $0\in\omega$. Given a function $E:\VR^d\rightarrow\VR$ let $K(x):=\int_\omega E(x-y)\;dy$, for $x\in\VR^d$. Then we have the following properties:
For dimension $d=2$ and $E(x):=\ln(|x|)$ there holds
\begin{itemize}
\item[(i)] $\|K-|\omega|E\|_{L^p(\VR^2)}<\infty$ for $p\in(2,\infty)$.
\item[(ii)] $\|\nabla(K-|\omega|E)\|_{L^p(\VR^2)^2}<\infty$ for $p\in(1,2)$.
\end{itemize}
For dimension $d>2$ and $E(x):=|x|^{-k}$, $k\in\VN$ there holds
\begin{itemize}
\item[(iii)] $\|K-|\omega|E\|_{L^p(\VR^d)}<\infty$ for $p\in(\frac{d}{k+1},\frac{d}{k})$.
\item[(iv)] $\|\nabla(K-|\omega|E)\|_{L^p(\VR^d)^d}<\infty$ for $p\in(\frac{d}{k+2},\frac{d}{k+1})$.
\end{itemize}
\end{lemma}

\begin{proof}
We start with the two dimensional case. Hence, let $E(x)=\ln(|x|)$. Note that a Taylor expansion shows that there is a $R>0$
\begin{equation}\label{eq:remainder_Lp_2}
K(x)-|\omega|E(x)=\int_\omega E(x-y)-E(x)\;dx\le C|x|^{-1},\quad\text{ for }|x|>R.
\end{equation}
Now splitting the integral with respect to this constant $R$ we get
\begin{equation}
\|K-|\omega|E\|^p_{L^p(B_R(0))}\le\|K\|^p_{L^p(B_R(0))}+\||\omega|E\|^p_{L^p(B_R(0))}<\infty\quad\text{ for all }p\in[1,\infty),
\end{equation}
since $\ln(|x|)\in L_p^{\text{loc}}(\VR^2)$ and the same holds for $K$. For the second part we use \eqref{eq:remainder_Lp_2} to conclude
\begin{equation}
\|K-|\omega|E\|^p_{L^p(B_R(0)^c)}\le C\int_{B_R(0)^c}|x|^{-p}\;dx=C\int_R^\infty r^{1-p}\;dr<\infty\quad\text{ for all }p\in(2,\infty),
\end{equation}
which shows (i). From \eqref{eq:remainder_Lp_2} we further see that
\begin{equation}\label{eq:remainder_Wp_2}
\nabla\left(K(x)-|\omega|E(x)\right)\le C|x|^{-2},\quad\text{ for }|x|>R.
\end{equation}
Now, noting that $|\nabla E(x)|=|x|^{-1}$, we see that
\begin{equation}\label{eq:estimate_Wp2_int}
\|\nabla E\|^p_{L^p(B_R(0))^2}=\int_0^Rr^{1-p}\;dr<\infty\quad\text{ for all }p\in(1,2).
\end{equation}
Since the same holds true for $K$, we conclude $\|\nabla(K-|\omega|E)\|_{L^p(B_R(0))^2}<\infty$ for all $p\in(1,2)$. For the exterior domain, we again use \eqref{eq:remainder_Lp_2} to conclude
\begin{equation}\label{eq:estimate_Wp2_ext}
\|\nabla\left(K-|\omega|E\right)\|^p_{L^p(B_R(0)^c)^2}\le C\int_{B_R(0)^c}|x|^{-2p}\;dx=C\int_R^\infty r^{1-2p}\;dr<\infty\quad\text{ for all }p\in(1,\infty).
\end{equation}
Combining \eqref{eq:estimate_Wp2_int} and \eqref{eq:estimate_Wp2_int} yields (ii). Similar arguments, exploiting the Taylor expansion of $|x|^{-k}$ for $k\in\VN$, shows item (iii) and (iv).
\end{proof}

\begin{lemma}\label{lem:scaling_norms_q}
For $x_0\in\Dsf$ and $\eps>0$ let $T_\eps(x):=x_0+\eps x$ and $\Dsf_\eps:=T_\eps^{-1}(\Dsf)$. Further define for $1<p<\infty$ the scaled norm
\begin{equation}
\|\varphi\|_{\eps,p}:=\eps\|\varphi\|_{L_p(\Dsf_\eps)}+\|\nabla\varphi\|_{L_p(\Dsf_\eps)^d},\quad\text{ for all }\varphi\in W^{1,p}(\Dsf_\eps).
\end{equation}
Then there holds:
\begin{itemize}
\item[(i)] $\|\varphi\circ T_\eps^{-1}\|_{W^{1,p}(\Dsf)}=\eps^{\frac{d}{p}-1}\|\varphi\|_{\eps,p}$.
\item[(ii)] $\|\varphi\circ T_\eps^{-1}\|_{L_p(\partial\Dsf)}=\eps^{\frac{d-1}{p}}\|\varphi\|_{L_p(\partial\Dsf_\eps)}$.
\item[(iii)] $|\varphi\circ T_\eps^{-1}|_{W^{\alpha,p}(\partial\Dsf)}=\eps^{\frac{d-1}{p}-\alpha}|\varphi|_{W^{\alpha,p}(\partial\Dsf_\eps)}$.
\item[(iv)] $\|\varphi\|_{\eps,p}\le C\left(\eps^{1-\frac{1}{p}}\|\varphi\|_{L_p(\partial\Dsf_\eps)}+|\varphi|_{W^{1-\frac{1}{p},p}(\partial\Dsf_\eps)}\right)$.
\end{itemize}
\end{lemma}

\begin{proof}
ad (i): This is a direct consequence of the scaling of $L_p$ norms.\newline
ad (ii): The same argument as before, considering $\text{dim}(\partial\Dsf)=d-1$.\newline
ad (iii): We have
\begin{align}
|\varphi\circ T_\eps^{-1}|_{W^{\alpha,p}(\partial\Dsf)}^p=&\int_{\partial\Dsf}\int_{\partial\Dsf}\frac{|\varphi\circ T_\eps^{-1}(x)-\varphi\circ T_\eps^{-1}(y)|^p}{|x-y|^{\alpha p+d-1}}\;dxdy\\
=&\eps^{2d-2}\int_{\partial\Dsf_\eps}\int_{\partial\Dsf_\eps}\frac{|\varphi(x)-\varphi(y)|^p}{|(x_0+\eps x)-(x_0+\eps y)|^{\alpha p+d-1}}\;dxdy\\
=&\eps^{d-1-\alpha p}\int_{\partial\Dsf_\eps}\int_{\partial\Dsf_\eps}\frac{|\varphi(x)-\varphi(y)|^p}{|x-y|^{\alpha p+d-1}}\;dxdy\\
=&\eps^{d-1-\alpha p}|\varphi|_{W^{\alpha,p}(\partial\Dsf)}^p.
\end{align}
ad (iv): Using the previous scalings and the right-inverse extension operator on $\Dsf$, we get
\begin{align}
\|\varphi\|_{\eps,p}=&\eps^{1-\frac{d}{p}}\|\varphi\circ T_\eps^{-1}\|_{W^{1,p}(\Dsf)}\\
\le&C\eps^{1-\frac{d}{p}}\left(\|\varphi\circ T_\eps^{-1}\|_{L_p(\partial\Dsf)}+|\varphi\circ T_\eps^{-1}|_{W^{1-\frac{1}{p},p}(\partial\Dsf)}\right)\\
=&C\eps^{1-\frac{d}{p}}\left(\eps^{\frac{d-1}{p}}\|\varphi\|_{L_p(\partial\Dsf_\eps)}+\eps^{\frac{d-1}{p}-(1-\frac{1}{p})}|\varphi|_{W^{1-\frac{1}{p},p}(\partial\Dsf_\eps)}\right)\\
=&C\left(\eps^{1-\frac{1}{p}}\|\varphi\|_{L_p(\partial\Dsf_\eps)}+|\varphi|_{W^{1-\frac{1}{p},p}(\partial\Dsf_\eps)}\right).
\end{align}
\end{proof}

\begin{lemma}\label{lem:scaling_bnd_transmission}
Let $g(x):=|x|^{-d}$, for $x\in\VR^d$. Then there holds
\begin{itemize}
\item[(i)] $\|g\|_{L_p(\partial\Dsf_\eps)}\le C\eps^{\frac{1-d}{p}+d}$.
\item[(ii)] $|g|_{W^{\alpha,p}(\partial\Dsf_\eps)}\le C\eps^{\frac{1-d}{p}+d+\alpha}$.
\end{itemize}
\end{lemma}

\begin{proof}
ad (i): We have
\begin{equation}
\|g\|_{L_p(\partial\Dsf_\eps)}^p=\int_{\partial\Dsf_\eps}|x|^{-dp}\;dx=\eps^{1-d}\int_{\partial\Dsf}\left|\frac{x-x_0}{\eps}\right|^{-dp}\;dx\le C\eps^{1-d+dp}.
\end{equation}
ad (ii): Similarly we conclude
\begin{align}
|g|_{W^{\alpha,p}(\partial\Dsf_\eps)}^p=&\int_{\partial\Dsf_\eps}\int_{\partial\Dsf_\eps}\frac{|g(x)-g(y)|^p}{|x-y|^{p\alpha+d-1}}\\
=&\eps^{2-2d}\int_{\partial\Dsf}\int_{\partial\Dsf}\frac{|g\circ T_\eps^{-1}(x)-g\circ T_\eps^{-1}(y)|^p}{|x-y|^{p\alpha+d-1}\eps^{-(p\alpha+d-1)}}\\
=&\eps^{1-d+p\alpha}\int_{\partial\Dsf}\int_{\partial\Dsf}\frac{|g\circ T_\eps^{-1}(x)-g\circ T_\eps^{-1}(y)|^p}{|x-y|^{p\alpha+d-1}}\\
\le&C\eps^{1-d+p\alpha}\eps^{-p+p(d+1)}\\
=&C\eps^{1-d+p(d+\alpha)},
\end{align}
where we used $|g\circ T_\eps^{-1}(x)-g\circ T_\eps^{-1}(y)|\approx\eps^{-1}|\nabla (g\circ T_\eps^{-1})(x)\cdot(x-y)|$ and $\nabla g(x)\approx |x|^{-(d+1)}$. For a more detailed proof of this estimate we refer to \cite[Lemma 4.4]{a_BAST_2021a}.
\end{proof}

\subsection*{Acknowledgements}
Idriss Mazari-Fouquer was partially supported by the project ``Analysis and simulation of optimal shapes - application to life sciences'' of the Paris City Hall and by the french ANR-18-CE40-0013-SHAPO on shape optimisation.\newline
Phillip Baumann has been funded by the Austrian Science Fund (FWF) project P 32911.

 \bibliographystyle{plain}
 \bibliography{special-issue}

\begin{thebibliography}{10}

\bibitem{b_AMKA_2007a}
H.~Ammari and H.~Kang.
\newblock {\em Polarization and moment tensors}, volume 162 of {\em Applied
  Mathematical Sciences}.
\newblock Springer, New York, 2007.
\newblock With applications to inverse problems and effective medium theory.

\bibitem{a_AM_2006a}
S.~Amstutz.
\newblock Sensitivity analysis with respect to a local perturbation of the
  material property.
\newblock {\em Asymptot. Anal.}, 49(1-2):87--108, 2006.

\bibitem{a_AM_2021a}
S.~Amstutz.
\newblock An introduction to the topological derivative.
\newblock {\em Engineering Computations}, September 2021.

\bibitem{a_BAGAST_2021b}
P.~Baumann, P.~Gangl, and K.~Sturm.
\newblock Complete topological asymptotic expansion for {$L_2$} and {$H^1$}
  tracking-type cost functionals in dimension two and three.
\newblock {\em submitted}, 2021.

\bibitem{a_BAST_2021a}
P.~Baumann and K.~Sturm.
\newblock Adjoint based methods for the computation of higher order topological
  derivatives with an application to linear elasticity.
\newblock {\em Engineering Computations}, 39(1), 2021.

\bibitem{a_BEDELAMA_2018a}
S.~Bertoluzza, A.~Decoene, L.~Lacouture, and S.~Martin.
\newblock Local error estimates of the finite element method for an elliptic
  problem with a dirac source term.
\newblock {\em Numer. Methods Partial Differential Equations}, 34(1):97--120,
  2018.

\bibitem{Boccardo}
L.~Boccardo and T.~Gallou{\"e}t.
\newblock Non-linear elliptic and parabolic equations involving measure data.
\newblock {\em Journal of Functional Analysis}, 87(1):149--169, 1989.

\bibitem{a_DE_2017a}
M.~C. Delfour.
\newblock Topological derivative: Semidifferential via {M}inkowski content.
\newblock {\em J. Convex Anal.}, 25(3), 2017.

\bibitem{a_DE_2022a}
M.~C. Delfour.
\newblock Topological derivative of state-constrained objective functions: A
  direct method.
\newblock {\em SIAM Journal on Control and Optimization}, 60(1):22--47, 2022.

\bibitem{c_DEST_2016a}
M.~C. Delfour and K.~Sturm.
\newblock Minimax differentiability via the averaged adjoint for control/shape
  sensitivity.
\newblock {\em IFAC-PapersOnLine}, 49(8):142--149, 2016.

\bibitem{a_DELI_1955a}
J.~Deny and J.~L. Lions.
\newblock Les espaces du type de {B}eppo {L}evi.
\newblock {\em Ann. Inst. Fourier, Grenoble}, 5:305--370 (1955), 1953--54.

\bibitem{a_ER_1985a}
K.~Eriksson.
\newblock Finite element methods of optimal order for problems with singular
  data.
\newblock {\em Math. Comp.}, 44(170):345--360, 1985.

\bibitem{b_EV_2010a}
L.~C. Evans.
\newblock {\em Partial differential equations}, volume~19 of {\em Graduate
  Studies in Mathematics}.
\newblock American Mathematical Society, Providence, RI, second edition, 2010.

\bibitem{a_GAST_2020a}
P.~Gangl and K.~Sturm.
\newblock A simplified derivation technique of topological derivatives for
  quasi-linear transmission problems.
\newblock {\em ESAIM Control Optim. Calc. Var.}, 26:Paper No. 106, 20, 2020.

\bibitem{a_GAST_2020b}
P.~Gangl and K.~Sturm.
\newblock Topological derivative for {PDE}s on surfaces.
\newblock {\em SIAM J. Control Optim.}, 60(1):81--103, 2022.

\bibitem{a_GAGUMA_2001a}
S.~Garreau, P.~Guillaume, and M.~Masmoudi.
\newblock The topological asymptotic for {PDE} systems: The elasticity case.
\newblock {\em SIAM Journal on Control and Optimization}, 39(6):1756--1778,
  2001.

\bibitem{b_HIULUL_2009a}
M.~Hinze, R.~Pinnau, M.~Ulbrich, and S.~Ulbrich.
\newblock {\em Optimization with {PDE} Constraints}.
\newblock Springer, New York, 2009.

\bibitem{a_IGNAROSOSZ_2009a}
M.~Iguernane, S.~Nazarov, J.-R. Roche, J.~Soko{\l}owski, and K.~Szulc.
\newblock Topological derivatives for semilinear elliptic equations.
\newblock {\em International Journal of Applied Mathematics and Computer
  Science}, 19(2), jan 2009.

\bibitem{b_ITKU_2008a}
K.~Ito and K.~Kunisch.
\newblock {\em Lagrange multiplier approach to variational problems and
  applications}.
\newblock Society for Industrial and Applied Mathematics, Philadelphia, PA,
  2008.

\bibitem{Littman}
W.~Littman, G.~Stampacchia, and H.~F. Weinberger.
\newblock Regular points for elliptic equations with discontinuous
  coefficients.
\newblock {\em Annali della Scuola Normale Superiore di Pisa - Classe di
  Scienze}, 17(1-2):43--77, 1963.

\bibitem{b_MANAPL_2012a}
V.~Maz'ya, S.~Nazarov, and B.~Plamenevskij.
\newblock {\em Asymptotic Theory of Elliptic Boundary Value Problems in
  Singularly Perturbed Domains: Volume I}.
\newblock Operator Theory: Advances and Applications. Birkhäuser Basel, 2012.

\bibitem{b_MANAPL_2012b}
V.~Maz'ya, S.~Nazarov, and B.~Plamenevskij.
\newblock {\em Asymptotic Theory of Elliptic Boundary Value Problems in
  Singularly Perturbed Domains Volume II: Volume II}.
\newblock Operator Theory: Advances and Applications. Birkhäuser Basel, 2012.

\bibitem{a_MEAPASC}
C.~Meyer, L.~Panizzi, and A.~Schiela.
\newblock Uniqueness criteria for the adjoint equation in state-constrained
  elliptic optimal control.
\newblock {\em Numer. Funct. Anal. Optim.}, 32(9):983--1007, 2011.

\bibitem{b_NOSO_2013a}
A.~A. Novotny and J.~Soko{\l}owski.
\newblock {\em Topological derivatives in shape optimization}.
\newblock Interaction of Mechanics and Mathematics. Springer, Heidelberg, 2013.

\bibitem{b_NOSOZO_2019a}
A.~A. Novotny, J.~Soko{\l}owski, and A.~{\.Z}ochowski.
\newblock {\em Applications of the topological derivative method}, volume 188
  of {\em Studies in Systems, Decision and Control}.
\newblock Springer, Cham, 2019.
\newblock With a foreword by Michel Delfour.

\bibitem{Ponce}
A.~C. Ponce.
\newblock Selected problems on elliptic equations involving measures.
\newblock {\em arXiv}, 2012.

\bibitem{a_SOZO_1999a}
J.~Soko{\l}owski and A.~Zochowski.
\newblock On the topological derivative in shape optimization.
\newblock {\em SIAM Journal on Control and Optimization}, 37(4):1251--1272,
  1999.

\bibitem{Stampacchia}
G.~Stampacchia.
\newblock Le probl\`eme de {Dirichlet} pour les \'equations elliptiques du
  second ordre \`a coefficients discontinus.
\newblock {\em Annales de l'Institut Fourier}, 15(1):189--257, 1965.

\bibitem{a_ST_2020a}
K.~Sturm.
\newblock Topological sensitivities via a {L}agrangian approach for semilinear
  problems.
\newblock {\em Nonlinearity}, 33(9):4310--4337, 2020.

\bibitem{b_TR_2010a}
F.~Tröltzsch.
\newblock {\em Optimal Control of Partial Differential Equations: Theory,
  Methods, and Applications}.
\newblock Graduate studies in mathematics. American Mathematical Society, 2010.

\end{thebibliography}
\end{document}